\documentclass[final,onefignum,onetabnum]{siamonline190516}

\usepackage{braket,amsfonts}

\usepackage{array}

\usepackage[caption=false]{subfig}
\usepackage{pgfplots}

\newsiamthm{claim}{Claim}
\newsiamremark{remark}{Remark}
\newsiamremark{hypothesis}{Hypothesis}
\crefname{hypothesis}{Hypothesis}{Hypotheses}

\usepackage{algorithmic}

\usepackage{graphicx,epstopdf}

\Crefname{ALC@unique}{Line}{Lines}

\usepackage{amsopn}

\usepackage{xspace}
\usepackage{bold-extra}
\usepackage[most]{tcolorbox}

\colorlet{texcscolor}{blue!50!black}
\colorlet{texemcolor}{red!70!black}
\colorlet{texpreamble}{red!70!black}
\colorlet{codebackground}{black!25!white!25}

\lstdefinestyle{siamlatex}{%
	style=tcblatex,
	texcsstyle=*\color{texcscolor},
	texcsstyle=[2]\color{texemcolor},
	keywordstyle=[2]\color{texemcolor},
	moretexcs={cref,Cref,maketitle,mathcal,text,headers,email,url},
}

\tcbset{%
	colframe=black!75!white!75,
	coltitle=white,
	colback=codebackground, %
	colbacklower=white, %
	fonttitle=\bfseries,
	arc=0pt,outer arc=0pt,
	top=1pt,bottom=1pt,left=1mm,right=1mm,middle=1mm,boxsep=1mm,
	leftrule=0.3mm,rightrule=0.3mm,toprule=0.3mm,bottomrule=0.3mm,
	listing options={style=siamlatex}
}

\DeclareTotalTCBox{\code}{ v O{} }
{ %
	fontupper=\ttfamily\color{black},
	nobeforeafter,
	tcbox raise base,
	colback=codebackground,colframe=white,
	top=0pt,bottom=0pt,left=0mm,right=0mm,
	leftrule=0pt,rightrule=0pt,toprule=0mm,bottomrule=0mm,
	boxsep=0.5mm,
	#2}{#1}

\patchcmd\newpage{\vfil}{}{}{}
\usepackage{amssymb}
\usepackage{graphicx,color,cite}

\usepackage{comment}

\usepackage{fourier, esint}
\usepackage[margin=1.5in]{geometry}
\usepackage{enumitem}
\usepackage{bbm}

\usepackage{tikz,pgfplots}
\usetikzlibrary{spy}
\usepackage{makecell}
\usepgfplotslibrary{fillbetween}
\usetikzlibrary{pgfplots.groupplots}
\usepackage{multirow}

\usepackage{mathtools}

\usepackage{cleveref}

\usepackage{apptools}
\usepackage{chngcntr}
\newtheorem{assumption}{Assumption}
\newtheorem{example}{Example}
\AtAppendix{\counterwithin{example}{section}}

\def\VV{\mathbb{V}}
\def\EE{\mathbb{E}}
\def\PP{\mathbb{P}}
\def\RR{\mathbb{R}}

\DeclareRobustCommand{\argmin}{\operatorname*{argmin}}

\def\EE{\mathbb{E}}\def\PP{\mathbb{P}}
\def\RR{\mathbb{R}}

\def\<{\langle} \def\>{\rangle}

\def\Gaussiancase{$\theta\sim\mathcal{N}(\theta_0,C)$}
\newcommand{\dthetaF}[1]{\nabla_\theta F(#1)}
\newcommand{\ddthetaF}[1]{\nabla^2_\theta F(#1)}
\newcommand{\dxiF}[1]{\nabla_\xi\tilde{F}(#1)}
\newcommand{\dxiI}[1]{\nabla_\xi\tilde{I}(#1)}
\newcommand{\ddxiF}[1]{\nabla^2_\xi\tilde{F}(#1)}

\newcommand{\halfC}{C^{\frac{1}{2}}}
\newcommand{\invA}{A^{-1}}

\newcommand{\charOmega}[1]{\mathbbm{1}_{\varOmega(z)}(#1)}
\newcommand{\Gauss}[2]{\sim\mathcal{N}(#1,#2)}

\newcommand{\xispace}{\tilde{\varOmega}}
\newcommand{\xihalfspace}{\tilde{\mathcal{H}}}	
\newcommand{\xiquaspace}{\tilde{\mathcal{Q}}}	
\newcommand{\xistar}{\xi^\star}
\newcommand{\nstar}{\hat{n}^\star}

\newcommand{\data}{PlotData/}

\newcommand{\xinorm}{\|\xi\|}
\newcommand{\xisnorm}{\|\xistar\|}
\newcommand{\eone}{e_1}
\newcommand{\eortho}{e_1^\perp}

\newcommand{\eospace}{E_1^\perp}
\newcommand{\musn}{\mu^{SN}}
\newcommand{\Hxi}{\xihalfspace_{\xistar}}
\newcommand{\thetastar}{\theta^\star}
\newcommand{\Qxi}{\xiquaspace_{\xistar}}

\newcommand{\Proj}{P_n} %

\newcommand{\halfQ}{C^{-\frac{1}{2}}}

\newcommand{\invC}{C^{-1}}

\newcommand{\numosamples}{$10^5$}

\newcommand{\cOne}{purple!10!blue!60!white}
\newcommand{\cTwo}{purple!40!blue!60!white}
\newcommand{\cThree}{purple!70!blue!60!white}
\newcommand{\cFour}{purple!100!blue!60!white}

\newcommand{\cMC}{blue!80}
\newcommand{\cMCci}{blue!50}
\newcommand{\cFit}{red!50!blue!80!black!70}
\newcommand{\cFORM}{yellow!80!red!70}
\newcommand{\cSORM}{red!90!black!70}
\newcommand{\cIS}{black!40!green!80}
\newcommand{\cISci}{black!40!green!50}

\newcommand{\vertiii}[1]{{\left\vert\kern-0.25ex\left\vert\kern-0.25ex\left\vert #1 
    \right\vert\kern-0.25ex\right\vert\kern-0.25ex\right\vert}}

  \title{Extreme event probability estimation using PDE-constrained
    optimization and large deviation theory, with application to
    tsunamis\thanks{\textcolor{blue!40!black}{Updated April 2025. Compared to the published
      version (CAMCoS 16(2), 2021), additional assumptions were needed in
      (and added to) \cref{lm:SO,thm:SO}.
      We thank Elisabeth Ullmann and Jules Pertinand for helpful discussions on this point.}
      \funding{S.~T.\ and G.~S.\ were partially supported by the US
        National Science Foundation (NSF) through grants DMS \#1723211
        and EAR \#1646337,
        and by the SciDAC program funded by the U.S.\ Department of
        Energy, Office of Science, Advanced Scientific Computing
        Research, and Biological and Environmental Research Programs.
        E.~V.-E. was supported in part by the NSF Materials Research
        Science and Engineering Center Program grant DMR \#1420073, by
        NSF grant DMS \#152276, by the Simons Collaboration on Wave
        Turbulence, grant \#617006, and by ONR grant
        \#N4551-NV-ONR.}}}

  \author{Shanyin Tong\thanks{Courant Institute, New York
      University, New York, USA (\email{shanyin.tong@nyu.edu},
      \email{eve2@cims.nyu.edu}, \email{stadler@cims.nyu.edu}).}
    \and Eric Vanden-Eijnden$^\dagger$
    \and Georg Stadler$^\dagger$}
	
\headers{Extreme events, PDE-constrained optimization and LDT}{Shanyin Tong, Eric Vanden-Eijnden and Georg Stadler}

\begin{document}

\maketitle

\begin{abstract}
 We propose and compare methods for the analysis of extreme events
in complex systems governed by PDEs that involve
random parameters, in situations where we are interested in
quantifying the probability that a scalar function of the system's
solution is above a threshold. If the threshold is large,
this probability is small and its accurate estimation is
challenging.  To tackle this difficulty, we blend theoretical
results from large deviation theory (LDT) with numerical tools from
PDE-constrained optimization. Our methods first compute parameters
that minimize the LDT-rate function over the set of parameters
leading to extreme events, using adjoint methods to compute the
gradient of this rate function. The minimizers give information
about the mechanism of the extreme events as well as estimates of
their probability. We then propose a series of methods to refine
these estimates, either via importance sampling or geometric
approximation of the extreme event sets. Results are formulated for
general parameter distributions and detailed expressions are
provided when Gaussian distributions. We give theoretical
and numerical arguments showing that the performance of our methods
is insensitive to the extremeness of the events we are interested
in.  We illustrate the application of our approach to quantify the
probability of extreme tsunami events on shore. Tsunamis are
typically caused by a sudden, unpredictable change of the ocean
floor elevation during an earthquake.  We model this change as a
random process, which takes into account the underlying physics.  We
use the one-dimensional shallow water equation to model tsunamis
numerically. In the context of this example, we present a comparison
of our methods for extreme event probability estimation, and find
which type of ocean floor elevation change leads to the largest
tsunamis on shore.
\end{abstract}


\section{Introduction}
Extreme events
tend to occur rarely but are often consequential when they do.
Examples from natural, social, and engineered systems include extreme
weather patterns such as hurricanes or tornadoes, pandemics, the
collapse of financial systems, cascading failures in power grids, and
structural damage in dams or bridges.  Estimating the probability of
these events and uncovering the mechanisms behind their emergence can
help inform strategies to mitigate their effects. However, given the
complexity of their dynamics, it is typically unfeasible to calculate
their probabilities explicitly. Monte Carlo methods are the standard
approach to studying complex systems that include
uncertainty. Unfortunately, these methods become inefficient to
explore the probability tails associated with extreme events.  The aim
of this paper is to design efficient methods to estimate tail
probabilities occurring in complex systems.

The methods we propose are meant to be generic and applicable to a
broad class of problems. However, in this paper we use tsunamis as our
main application example.  Tsunami waves are generated by the
displacement of a large amount of water due to a sudden and
unpredictable elevation change in the ocean floor. This change, which
occurs in conjunction with an earthquake, typically happens tens or
hundreds of kilometers away from the coast line. As the tsunami waves
travel to shore, they speed up in the deeper parts of the ocean and
slow down in the shallower parts. This nonlinear interaction with the
ocean floor combined with reflections from land features shape the
tsunami waves that eventually reach the shore.  To quantify the
flooding-induced damage in locations of interest (e.g., cities or
critical infrastructure), we use the average tsunami wave height in
regions close to those locations.  The random component in this system
is the ocean floor elevation change. Given a distribution for possible
elevation changes, we study the probability of observing extreme
tsunamis close to the locations of interest. Additionally, we explore
which type of elevation changes result in the largest tsunamis.  The
next section summarizes our approach, prior to a review of related
work in this area.

\subsection{Mathematical setup and methodological aspects}
\label{sec:mathsetup}

Following the strategy proposed
in~\cite{dematteis2018rogue,dematteis2019extreme}, we use tools from
large deviation theory (LDT) to connect probability estimation of
extreme events with optimization. We assume that the randomness of the
event under consideration can be captured by a parameter $\theta$
taking values in a Hilbert space $\varOmega$, e.g., $\varOmega=\RR^n$
or $\varOmega=L^2(\mathcal D)$ for a domain $\mathcal D\subset \RR^n$,
and whose statistics is specified by a probability measure~$\mu$.
Given a parameter-to-event map $F:\varOmega\to \RR$ such that the
larger $F(\theta)$, the rarer the event, we are interested in the
probability
\begin{equation}
\label{eq:LDT-probability}
P(z):=\mathbb{P}(F(\theta)\geq z), 
\end{equation}
when $z$ is large and hence $P(z)\ll 1$. In the applications we are
interested in, $F(\theta)$ is of the form $F(\theta)= G(u(\theta))$,
where $G$ is some functional evaluated on the solution $u$ of a
(partial) differential equation (PDE), which we will denote by
$e(u,\theta) = 0$: the parameter $\theta$ may enter this PDE for
instance as a forcing, or as boundary or initial condition, and
therefore its solution implicitly depends on $\theta$, $u=u(\theta)$.

We will show that computation of the probability
in~\eqref{eq:LDT-probability} is aided by finding the most likely
point (in the physical literature called \emph{instanton})
$\theta^\star(z)$ in the extreme event set
$\varOmega(z):=\{\theta\in \varOmega:F(\theta)\geq z\}$, i.e., the
solution of
\begin{equation}
\label{eq:LDT-instanton}
\theta^\star(z)=\argmin\limits_{\theta\in\varOmega(z)} I(\theta),
\end{equation}
where $I$ is the rate function from LDT defined in the subsequent
sections and $\theta^\star(z)$ is the global minimizer of $I$ over the
set $\varOmega(z)$, which we assume to be unique. When
$F(\theta) = G(u(\theta))$ where $u$ solves $e(u,\theta) = 0$,
\eqref{eq:LDT-instanton} has the form of a PDE-constrained
optimization problem. Under suitable assumptions on $F$ and the
distribution of $\theta$ to be detailed in \cref{sec:LDT}, the
minimum $\theta^\star(z)$ is attained on the boundary of
$\varOmega(z)$ and it can equivalently be characterized as solution of
the problem
\begin{equation}\label{eq:LDT-H}
\theta^\star(z)=\argmin_{\theta\in\varOmega} I(\theta)-\lambda F(\theta)
\end{equation}
for a specific parameter $\lambda>0$. A
variant of LDT then states that
\begin{equation}
\label{eq:LDT-principle}
\log P(z)\approx -I(\theta^\star(z)) \qquad \text{as } z\to\infty, 
\end{equation}
where ``$\approx$'' means that the ratio between the left and the
right sides goes to 1 as $z\to\infty$.  This shows that, by solving
optimization problems of the form \eqref{eq:LDT-instanton} (or
equivalently \eqref{eq:LDT-H} with appropriate $\lambda>0$), we can
estimate the log-asymptotic behavior of the probability $P(z)$ via
\eqref{eq:LDT-principle}.  The details, along with the assumptions
needed for~\eqref{eq:LDT-principle} to hold, are given in
\cref{sec:LDT}.

The next question we will address is how to get estimates of the
probability~\eqref{eq:LDT-probability} that are more accurate
than~\eqref{eq:LDT-principle}. We show that this can be done in
two ways. In \cref{sec:PE} we first propose an importance sampling
(IS) method based on the optimizers $\theta^\star(z)$ for
different~$z$. Compared to a vanilla Monte Carlo sampler, the sample
variance of this IS does not include the term
$\exp(-I(\theta^\star(z)))$. This is a significant improvement as this
term grows exponentially with the extremeness of events. This IS
method allows asymptotically exact computation of $P(z)$.

The second way to improve upon~\eqref{eq:LDT-principle} is to obtain
an estimate that holds without the logarithm in this equation. That is,
in \cref{sec:PFS}, we discuss how to find a function
$C_0:\RR \to (0,\infty)$ such that
\begin{equation}
\label{eq:PE-prefactor}
P(z)\approx C_0(z)\exp(-I(\theta^\star(z))), \qquad \text{as } z\to\infty.
\end{equation}
The function $C_0(z)\ge0$ is usually referred to as a ``prefactor''.
We will show that $C_0(z)$ can be calculated by exploiting the local
derivative information at the optimizer $\theta^\star(z)$ to construct
the second-order approximation of the extreme set boundary
$\partial\varOmega(z)$. In the engineering literature, this approach is
refereed to as Second Order Reliability Method (SORM), and in
\cref{sec:PFS} we discuss conditions under which SORM is
asymptotically exact, i.e., it leads to a prefactor $C_0(z)$ such
that~\eqref{eq:PE-prefactor} holds. Additionally, we show how low-rank
approximations can be used to compute SORM-based probabilities in high
parameter dimensions.
For completeness, in
\cref{sec:FORM} we review another approach used by
engineers, termed First Order Reliability Method (FORM), which gives
another expression for $C_0(z)$: the FORM expression for $C_0(z)$ is
simpler than that of SORM but we show that it is not
asymptotically exact in general.

As an illustration, in \cref{sec:AP,sec:results} we apply our
methodology to estimate the probability of extreme tsunami events on
shore, which are caused by random, earthquake-induced elevation
changes of the ocean floor described above.  Here, the
parameter-to-event map $F$ involves the solution of a system of
nonlinear PDEs, namely the shallow water equations. Since the random
parameter $\theta$ in this problem is high-dimensional, solving the
optimization problem \eqref{eq:LDT-instanton} is challenging. We use
an adjoint method for the efficient computation of derivatives of $F$
with respect to $\theta$ and discuss the challenges of the resulting
PDE-constrained optimization problem.

\subsection{Related literature}
Most methods for extreme event estimation are based on Monte Carlo
(MC), Markov Chain Monte Carlo (MCMC) or importance sampling (IS)
\cite{liu2008monte}. Standard MC sampling becomes impractical for
extreme events due to the large number of required samples for
unlikely events.  MCMC sampling have similar shortcomings, but
tailored variants such as Umbrella Sampling
\cite{torrie1977nonphysical} can improve the estimation of tail
probabilities. Importance sampling, \cite{kahn1953methods,
  bucklew2013introduction}, decreases the required number of samples
by using proposal distributions that reduce the variance of the
estimator. Recently proposed IS methods use ideas from Bayesian
inference to find a maximum a posterior (MAP) point and construct a
Gaussian distribution centered at that point as IS proposal
\cite{rao2020efficient, wahalbimc, sapsis2020output}. These methods
require MAP points that lie in the pre-image of certain extreme
events, and finding such events can be computationally extensive. In
particular, the authors of \cite{wahalbimc} compute a Gaussian IS
proposal by minimizing the Kullback-Leibler divergence to the ideal IS
proposal.  In \cite{rao2020efficient}, the authors propose to draw
observation pairs from Rice's formula. Both methods rely on the
linearity of the parameter-to-event maps and linearize them for
nonlinear problems.

In this paper, we follow the approach proposed
in~\cite{dematteis2019extreme} that takes the perspective of large
deviation theory \cite{dembo1998large,varadhan1984large} to estimate
extreme event probabilities in system with random components and
applies the resulting methods to quantify the probability of the
occurrence of rogue
waves~\cite{dematteis2018rogue,dematteis2019experimental}.  These
papers solve an optimization problem that finds the most important point
(also called instanton) in the extreme event set.  This present paper
uses a similar approach but generalize it in various directions, e.g.,
it provides prefactor estimators.  In a related approach, the authors
of \cite{farazmand2017variational,sapsis2018new} search for initial
condition leading to the highest growth in flow problems. This also
requires solution of an optimization problem related to LDT
optimization.

Probability estimation of extreme events is also of importance in
engineering, e.g., for assessing the structural reliability of
buildings or bridges \cite{ditlevsen1996structural}. Methods used in
this context are based on the point with largest probability density
(typically of a Gaussian distribution), combined with extreme event
set approximations called First and Second Order Reliability Methods
(FORM and SORM)
\cite{du2001most,rackwitz2001reliability,schueller1987critical}. These
methods use a truncated Taylor expansion of the parameter-to-event map
at the most probable point to estimate probabilities. Also IS methods
based on the most probable point have been proposed
\cite{kahn1953methods,schueller1987critical}.  Our approach has
similarities with these engineering methods, but uses instead the
minimizer of the rate function from LDT, which describes the
asymptotic behavior of the probability and can be used to design IS
methods~\cite{dupuis2004importance,vanden2012rare} . Since the rate
function of a Gaussian distribution is a multiple of its log-density,
our methods generalize FORM and SORM, and provide theoretical
justification for these approaches.  Moreover, our methods apply to
complicated dynamical systems (governed for instance by ODEs or PDEs)
with high-dimensional parameters as they only require derivatives that
can be computed efficiently using adjoint methods.

We use the methods we propose to estimate the probability of extreme
tsunami waves on shore after sudden earthquake-induced ocean floor
changes, which are modeled as random.  As governing equations, we use
the one-dimensional shallow water equations
\cite{leveque2008high,vreugdenhil2013numerical}, discretized with
discontinuous-Galerkin finite elements \cite{hesthaven2007nodal}. To
prevent the occurrence of shocks in these nonlinear hyperbolic
equations, we add artificial viscosity \cite{chen2013adaptive}. This
also provides justification to using the adjoint method to compute
derivatives for optimization problems governed by hyperbolic equations
\cite{giles2010convergence,ulbrich2019numerical}.

The proposed methods require the solution of optimization problems
involving complex systems that are typically governed by PDEs. While
the structure of these problems is similar to problems occurring in
optimal control and inverse problems, the extreme event perspective
suggests several novel research directions. First, it motivates to
study new classes of governing equations, e.g., hyperbolic systems and
their discretization \cite{wilcox2015discretely, hager2000runge,
  hajian2019total, giles2010convergence,
  ulbrich2019numerical}. Second, it required to study and compute
post-solution properties of minimizers, e.g., estimation of second
derivatives as in Bayesian inference
\cite{Bui-ThanhGhattasMartinEtAl13, AlexanderianPetraStadlerEtAl14} or
parametric sensitivity analysis \cite{Griesse04}. Third, as it is
typically unknown when an extreme event will occur, it motivates
further study of time-optimal control problems and their numerical
solution in complex applications \cite{Fattorini05,KunischRund15}.

\subsection{Contributions and limitations}
The main contributions of this work are as follows: (1) We present an
extreme event probability estimation framework that exploits
connections between probability estimation and PDE-constrained
optimization, and apply it to a complex example problem.  (2) We
propose approaches to refine the asymptotic probability estimates from
LDT by approximations of the extreme event sets. The computational
cost of these approximations is independent of the extremeness of the
events.  (3) We show that importance sampling leveraging the LDT
optimizers can lead to an exponential reductions of relative errors in
all parameter directions.  (4) As our tsunami application problem is
governed by the 1D shallow water equations, we derive adjoint
equations for this nonlinear hyperbolic conservation law and use them
to efficiently compute gradients of the LDT objective.

Our work also has several limitations: (1) Most of the presented
expressions for extreme event probability estimation are for an
underlying Gaussian probability distribution. Possible generalizations
depend on the probability measure and must be considered on a
case-by-case basis. However, our explicit expressions apply to
distributions that can be mapped to Gaussian distributions.  (2) The
proposed approach requires regularity properties, e.g., that the
optimization problem has a unique solution and that the rate function
of the parameter distribution is well-defined. Some properties of the
parameter-to-event map $F$ discussed in the next section can be
difficult to verify a priori, but some may be verified a posteriori.
(3) The tsunami model used in this work is one-dimensional, thus not
allowing some of the complexity of a more realistic two-dimensional
setup. However, our framework is generic and applies to more complex
problems as long as derivatives of the objective with respect to the
parameters are available.  (4) We make some simplifying choices in the
numerical scheme used for the shallow water equations, e.g., we use
uniform time steps and a global Lax-Friedrichs flux. Some of these
choices could be relaxed and while such a discussion is definitely
interesting, it is beyond the scope of this paper.

\subsection{Notation}\label{sec:notation}
Throughout the paper we repeatedly use asymptotic
estimates. Thus, we introduce the following notation, where we
consider the asymptotic parameter $s\to\infty$. Then, for $a(s), b(s)
> 0$, we introduce the notation:
\begin{subequations}\label{eq:notation}
	\begin{align}
	a(s)\approx b(s)  \quad  &\text{if} \quad \frac{a(s)}{b(s)} \to 1
	\text{ as } s\to\infty, \\
	a(s)\lesssim b(s)  \quad  &\text{if} \quad a(s)\le b(s) \text{
		for all } s, \text{ and } \frac{a(s)}{b(s)} \to 1 \text{ as } s\to \infty,\\
	a(s)\gtrsim b(s)  \quad  &\text{if} \quad a(s)\geq b(s) \text{
		for all } s, \text{ and } \frac{a(s)}{b(s)} \to 1 \text{ as } s\to \infty.
	\end{align}
\end{subequations}
We commonly use multivariate Gaussian parameters in $\RR^n$, $n\ge
1$. We say that a parameter $\theta$ follows \Gaussiancase~when
$\theta$ is a multivariate Gaussian parameter with mean $\theta_0\in
\RR^n$ and covariance matrix $C\in \RR^{n\times n}$. Here, $C$ is
assumed to be symmetric and positive definite.

We regularly use a Hilbert space $\varOmega$ and denote the
corresponding inner product by $\langle\cdot\,,\cdot\rangle$ and the
induced norm by $\|\cdot\|$. For the Euclidean inner product, we also
use the vector notation $a^\top b = \langle a,b\rangle$ whenever convenient.
Given a symmetric positive operator $Q$
on $\varOmega$, we denote the weighted inner product by
$\langle\cdot\,,\cdot\rangle_Q:=\langle \cdot\,,Q\cdot\rangle$ and the
induced norm by $\|\cdot\|_Q$.

\section{Large deviation theory and optimization}\label{sec:LDT}

Extreme event quantification aims at estimating the probability that a
certain scalar quantity, which is a function of a random parameter
$\theta$, is at or beyond a threshold.  In this section we summarize how
ideas from LDT can be used to establish a formal connection between
estimation of extreme events and optimization, loosely
following~\cite{dematteis2019extreme}. We first show how the
underlying distribution for the parameter $\theta$ defines the rate
function $I:\varOmega\mapsto \RR\cup\infty$ occurring in the optimization problem
\eqref{eq:LDT-instanton}.

For a parameter $\theta$ with probability
distribution $\mu(\theta)$, the cumulant generating function $S(\eta)$
is the logarithm of the moment generating function of $\theta$
\begin{equation}
\label{eq:LDT-S}
S(\eta)=\log\mathbb{E}e^{\langle\eta,\theta\rangle}=\log\int_{\varOmega}e^{\langle\eta,\theta\rangle}d\mu(\theta),
\end{equation}
and we define $I:\Omega\to \RR$ to be the Legendre transform of $S(\eta)$: 
\begin{equation}
\label{eq:LDT-I}
I(\theta) = \max_{\eta\in \varOmega}\left(\langle\eta,\theta\rangle-S(\eta)\right).
\end{equation}
We will be interested in problems in which $I(\theta)$ plays the role
of the large deviation rate function, as obtained from Gartner-Ellis
theorem when it applies~\cite{dembo1998large}, and will therefore
refer to it as such.  We note that $I(\theta)$ is convex by
definition, and it can be computed explicitly for some
distributions. For completeness, we derive it for multivariate
Gaussian and exponential distributions in \cref{sec:rate-examples}. In
particular, we find that the rate function $I(\theta)$ of a
multivariate Gaussian distribution is the negative log-probability
density. Next, we present the principle that allows to relate
constrained optimization over $I(\theta)$ to estimating probabilities.

\subsection{Large deviation principle} Given a parameter
$\theta\in \varOmega$ with probability measure $\mu$, and a
parameter-to-event map $F:\theta\mapsto \RR$, LDT relates the
probability $P(z)=\mathbb{P}(F(\theta)\geq z)$ and the minimizer
\eqref{eq:LDT-instanton} of the LDT rate function $I(\theta)$ in
\eqref{eq:LDT-I}. A sketch of this relation is shown in
\Cref{fig:LDTsketch}. We now provide a formal proof of the LDT result
\eqref{eq:LDT-principle}. This proof is based on the five assumptions
in \cite{dematteis2019extreme}, which we recall and generalize to
accommodate a more general class of extreme events sets $\varOmega(z)$
(see \Cref{th:as4}). Moreover, we discuss what each assumption means
for a multivariate Gaussian parameter distribution.

\begin{figure}[tb]\centering
	\begin{tikzpicture}[]
	\begin{axis}[compat=1.11, width=9cm, height=7cm,
	xmin=-5,
	xmax=4,
	ymin=-3,
	ymax=4,
	axis line style={draw=none},
	tick style={draw=none},
	yticklabels={,,},
	xticklabels={,,},
	]

	\draw [green!60!black,very thick] (-3,-0.78) -- (-3.2,-1.27);
	\draw [green!60!black,very thick] (-3,-0.78) .. controls (-2.8,-0.29) and (-2.2,-0.4) .. (-2,0);
	\filldraw [green!60!black] (-3,-0.78) circle (2pt);
	\filldraw [green!60!black] (-2,0) circle (2pt);		
	\filldraw [fill=red!8,draw=none]
	(4,-3) .. controls (2,0) and (1,-1)
	.. (0,0).. controls (-1,1) and (-0.9,2) .. (-1,4) -- (4,4)  -- (4,-3);

	\filldraw [green!60!black] (0,0) circle (2pt) node[yshift=-0.3cm,xshift=-0.1cm]{$\theta^\star(z)$};
	\draw [black, very thick, ->] (0,0)	-- (1,1) node[yshift=0.2cm,xshift=-0.4cm] {$\hat{n}^\star(z)$};
	
	\draw [green!60!black,very thick] (-2,0) .. controls (-1.7,0.6) and (-0.5,-0.5) .. (0,0);
	\filldraw [green!60!black] (0,0) circle (2pt);

			\filldraw [fill=red!8,draw=none]
	(2.5,-3) .. controls (2,0) and (1,-1)
	.. (0,0).. controls (-1,1) and (-0.9,2) .. (-1,4) -- (4,4)  -- (4,-3);
	\draw[red!60!black,very thick] (-1,4)  .. controls (-0.9,2)and (-1,1)
	.. (0,0).. controls (1,-1) and (2,0) .. (2.5,-3) node[near end, sloped, below]
	{$\partial \varOmega(z)$};
	
	\draw [green!60!black,very thick] (0,0) .. controls (1,1) and (1,-0.37) .. (2,0.45);
	\filldraw [green!60!black] (2,0.45) circle (2pt);

	\draw [green!60!black,very thick] (2,0.45)	-- (2.5,0.86) ;
	
	\draw [red!60!black,very thick, dashed] (0,-3) .. controls (0,-2) and (-2,-1.19) .. (-3,-0.78).. controls (-4,-0.37) and (-4,0) .. (-5,1) node[red!60!black,xshift=1.5cm,yshift=2cm] {level sets of $F(\theta)$};

	\draw [red!60!black,very thick, dashed] (1,-7) .. controls (4,-2) and (-1,-0.5) .. (-2,0).. controls (-3,0.5) and (-3,3) .. (-4,4);

	\draw [red!60!black,very thick,dashed] (4,-2) .. controls (3,0) and (3,-0.77) .. (2,0.45).. controls (1,1.67) and (2,2) .. (2,4);
	
	\draw[black, yshift=3.5cm,xshift=6.3cm] node {\textcolor{red!60!black}{$\varOmega(z)$}};
	\draw [blue,very thick,dotted] (-4,-2) ellipse (2 and 1.41);
	\draw [blue,very thick,dotted] (-4,-2) ellipse (3.46 and 2.45);
	\draw [blue,very thick,dotted] (-4,-2) ellipse (4.9 and 3.46) node [xshift=0.6cm, yshift=3cm] {level sets of $I(\theta)$};
	\draw [blue,very thick,dotted] (-4,-2) ellipse (6.93 and 4.9);
	\draw [blue,very thick,dotted] (-4,-2) ellipse (8.49 and 6);
	\end{axis}
	\end{tikzpicture}
	\caption{2D illustration of level sets of the rate function
		$I(\cdot)$ and the extreme event sets $\varOmega(z)$. For fixed
		$z$, $\theta^\star(z)$ is the solution to an optimization
		problem and thus the gradients $\nabla F(\theta^\star(z))$
		and $\nabla I(\theta^\star(z))$ align and after
		normalization equal to $\hat{n}^\star(z)$. The path of the
		optimizers $\thetastar(z)$ for different $z$ plays an important
		role in large deviation theory.
	}\label{fig:LDTsketch}
\end{figure}
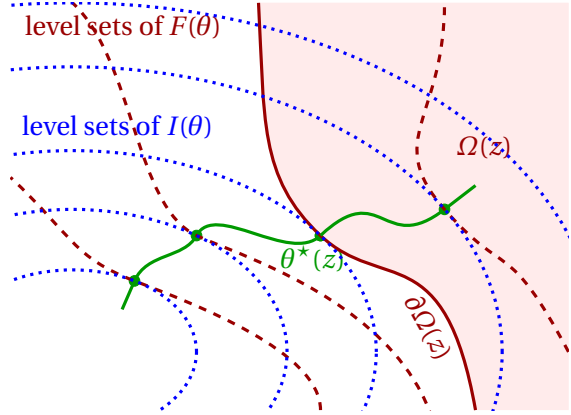

\begin{assumption}
  \label{th:as1} There exists a finite $z_0$ such that the restriction
  of the map $F$ to the preimage of the interval
  $(z_0,\infty)\subset\RR$, i.e., to the set
  $F^{-1}((z_0,\infty))\subset\varOmega$, is differentiable with
  $\|\nabla F\|\ge K >0$ for a suitable $K>0$.
\end{assumption}

\begin{assumption} 
  \label{th:as2} The probability measure $\mu$ is such that the
  cumulant generating function $S(\eta)$ \eqref{eq:LDT-S} exists for
  all $\eta\in \varOmega$ and defines a differentiable function
  $S:\varOmega \to \RR$.
\end{assumption}

For a Gaussian parameter, $S(\eta)=\eta^\top
\theta_0+\frac{1}{2}\eta^\top C\eta$ as shown in
\eqref{eq:LDT-Gaussian S}, and thus this assumption is automatically
satisfied.
As in \cite{dematteis2019extreme}, \Cref{th:as2} allows us to
introduce the tilted measure $d\mu_\eta(\theta)$, which is used in the
following assumptions:
\begin{equation}
\label{eq:LDT-tilted measure}
d\mu_\eta(\theta) = \frac{e^{\<\eta, \theta\>} d\mu(\theta)
}{\int_\mathcal{D} e^{\<\eta,  \theta\>} d\mu(\theta)} 
= e^{\<\eta, \theta\>-S(\eta)} d\mu(\theta)\,.
\end{equation}

\begin{assumption}
	\label{th:as3}
	There exists a finite $z_0$ and a constant $K$ such that,
        $\forall z\ge z_0$, the rate function $I(\theta)$ has the
        unique global minimizer $\thetastar(z)$ in the set
        $\varOmega(z)=\{\theta\in \varOmega:F(\theta)\geq z\}$. In
        addition, the map $\theta^\star:[z_0,\infty) \to \varOmega$ is
        continuously differentiable and $I(\theta^\star(\cdot))$ is
        strictly increasing with $z$ with
	\begin{equation}
	\label{eq:LDT-I assumpt}
	I(\theta^\star(z))\to \infty \quad \text{and} \quad \|\nabla
	I(\theta^\star(z))\|\ge K >0
	\quad \text{as}\quad z\to\infty.
	\end{equation}
\end{assumption}

For a Gaussian parameter,
$I(\theta)=\frac{1}{2}(\theta-\theta_0)^\top \invC(\theta-\theta_0)$,
so $I(\theta^\star(z))\to \infty$ as long as
$\|\theta^\star(z)\|\to\infty$ as $z\to\infty$. Additionally,
$\|\nabla
I(\theta^\star(z))\|=\|\invC(\theta^\star(z)-\theta_0)\|\geq\|(\theta^\star(z)-\theta_0)\|/\lambda_{\max}(\invC)\geq
K>0$ as long as
$\|(\theta^\star(z)-\theta_0)\|\geq K\lambda_{\max}(\invC)$ for
$z\geq z_0$, where $\lambda_{\max}(C)$ is the largest eigenvalue of
$C$. Thus, \Cref{th:as3} is satisfied when
$\|\theta^\star(z)\|\to\infty$ as $z\to\infty$.

Since the rate function $I$ is convex, \Cref{th:as3} implies that $\theta^\star(z) \in \partial \varOmega(z)$
for $z>z_0$, i.e., we can replace~\eqref{eq:LDT-instanton} with
\begin{equation}
\label{eq:LDT-instanton bdry}
\theta^\star(z) = \argmin_{\theta\in \partial\varOmega(z)} I(\theta)\,.
\end{equation}
The corresponding Euler-Lagrange equation is
\begin{equation}
\label{eq:LDT-EL}
\nabla I(\theta^\star(z)) = \lambda \nabla F(\theta^\star(z)),
\end{equation}
for some Lagrange multiplier $\lambda\in\RR$. Following
\cite{dematteis2019extreme}, if we define
$\eta^\star(z) := \nabla I(\theta^\star(z))$, it is easy to see that
the mean of $\mu_{\eta^\star(z)}$ is $\theta^\star(z)$. From the
Legendre transform, this implies that
$\<\eta^\star(z), \theta^\star(z)\> - S(\eta^\star(z)) =
I(\theta^\star(z))$.  Thus, we obtain an exact representation formula
for the probability $P(z)$:
\begin{equation}
\label{eq:LDT-P(z) tilted measure}
\begin{aligned}
P(z)&= \int _{\varOmega(z) } d\mu(\theta)
= \int _{\varOmega(z) } e^{S(\eta^\star(z))-
	\<\eta^\star(z),
	\theta\> } d\mu_{\eta^\star(z)}(\theta) \\
&= e^{-I(\theta^\star(z))}\int _{\varOmega(z) } e^{-
	\<\eta^\star(z), \theta-\theta^\star(z)\> }
d\mu_{\eta^\star(z)}(\theta)\,.
\end{aligned}
\end{equation}

To prove the large deviation principle \eqref{eq:LDT-principle}, we also need assumptions on $\varOmega(z)$. Differently from \cite{dematteis2019extreme}, we avoid the assumption that $\varOmega(z)$ is contained in the half-space 
\begin{equation}
\label{eq:LDT-H(z)}
\mathcal{H}(z) :=\left\lbrace  \theta : \left\langle  \hat{n}^\star(z), \theta- \theta^\star(z)\right\rangle \ge0\right\rbrace,
\end{equation}
where
$\hat n^\star (z)= \nabla F(\theta^\star(z))/\|\nabla
F(\theta^\star(z))\|=\nabla I(\theta^\star(z))/\|\nabla
I(\theta^\star(z))\|=\eta^\star(z) / \|\eta^\star(z)\|$. Instead, we
make a more general assumption.

\begin{assumption} 
	(Modified version of \cite{dematteis2019extreme}.)
	\label{th:as4}
	The set $\varOmega(z)$ satisfies 
	\begin{equation}
	\label{eq:LDT-ass4}
	\lim\limits_{z\to\infty}\dfrac{\log\left( \int _{\varOmega(z) } e^{-
			\<\eta^\star(z), \theta-\theta^\star(z)\> }
		d\mu_{\eta^\star(z)}(\theta)\right) }{I(\theta^\star(z))}\leq 0.
	\end{equation}
\end{assumption}
This assumption relaxes the condition that $\varOmega(z)$ is included
in $\mathcal{H}(z)$, and expresses that the measure of
$\varOmega(z)\backslash\mathcal{H}(z)$ must be sufficiently small.

For a Gaussian parameter \Gaussiancase, this assumption is related to the
half-space approximation discussed later in this paper. Namely,
the approximation \eqref{eq:PE-FORM Gauss} derived in \cref{sec:FORM}
implies
\begin{equation}
\label{eq:LDT-half-space}
\begin{aligned}
\int_{\mathcal{H}(z) } \!\!\!\!\!e^{-
	\<\eta^\star(z), \theta-\theta^\star(z)\> }
d\mu_{\eta^\star(z)}(\theta)
& =(2\pi)^{-n/2}\det(C)^{-1/2}\int _{\mathcal{H}(z) } \!\!\!\!\!e^{I(\theta^\star(z))-I(\theta)}d\theta\\
&\leq\frac{1}{\sqrt{4\pi I(\theta^\star(z))}}.
\end{aligned}
\end{equation} 
Thus, we only need that
\begin{equation}
\label{eq:LDT-ass4 Gaussian}
\lim\limits_{z\to\infty}\dfrac{\log\left( \frac{1}{\sqrt{4\pi I(\theta^\star(z))}}+\int _{\varOmega(z)\backslash\mathcal{H}(z) } e^{-
		\<\eta^\star(z), \theta-\theta^\star(z)\> }
	d\mu_{\eta^\star(z)}(\theta)\,\right) }{I(\theta^\star(z))}\leq 0,
\end{equation}
which means that the part of $\Omega(z)$ not contained in
$\mathcal{H}(z)$ must be sufficiently small. As further discussed in
\cref{sec:SORM} later in this paper, if the set $\varOmega(z)$ is
contained in a paraboloid centered at $\thetastar(z)$, the curvature
of that paraboloid must be in proper relation to the quadratic rate
function. For details, we refer to the proof of 
\cref{thm:SO}.

For the next assumption, which is needed for the lower bound, we first
define $G(z,s)  :=  \mu_{\eta^\star(z)}\left(\varOmega(z)\setminus \mathcal{H} (z,s)\right)$ with 
\begin{equation}
\label{eq:LDT-H(z,s)}
\mathcal{H} (z,s) := \left\{ \theta \ : \ \<\hat n^\star(z), \theta - \theta^\star(z) -
\hat n^\star(z) s\> \ge 0\right\}.
\end{equation}

\begin{assumption}
	\label{th:as5} There exists $s_1>0$ such that
	\begin{equation}\label{eq:10ass}
	\lim_{z\to\infty} \frac{\log G(z, s_1) 
	}{I(\theta^\star(z))} = 0.
	\end{equation}
\end{assumption} 
This  assumption ensures that the shape of $\varOmega(z)$
does not degenerate as $z\to\infty$.

\begin{theorem}[Large deviation principle]\label{th:1}
	Under \Cref{th:as1} -- \Cref{th:as5},  the following
	result, which is equivalent to \eqref{eq:LDT-principle}, holds.
	\begin{equation}
	\label{eq:LDT}
	\lim_{z\to\infty} \frac{\log P(z) }{I(\theta^\star(z))} =
	\lim_{z\to\infty} \frac{\log\mu(\varOmega(z)) }{I(\theta^\star(z))} =
	-1.
	\end{equation}
\end{theorem}
\noindent
We note that this theorem is slightly different from a standard LDP \cite{broniatowski1995tauberian,borovkov1965multi,dembo1998large,varadhan1984large}
since it involves taking the limit of the ratio of $\log P(z)$ and
$I(\theta^\star(z))$: in contrast  a standard LDP would also establish how
$I(\theta^\star(z))$ grows as $z\to\infty$. Our result does not give
this growth explicitly, and it has to be calculated numerically via
estimation of $I(\theta^\star(z))$  for large $z$. We will explain how
to do so in \cref{sec:PE}.

\begin{proof}
	\Cref{th:as1}--\Cref{th:as3} allow us to introduce the tilted
	measure and other terms discussed above. Applying
	\Cref{th:as4} to \eqref{eq:LDT-P(z) tilted measure}, we find
	an upper bound for $P(z)$, namely
	\begin{equation}
	\label{eq:LDT-P(z) up bd}
	\begin{aligned}
	\lim\limits_{z\to\infty}\dfrac{\log P(z)}{I(\theta^\star(z))}
	=-1+\lim\limits_{z\to\infty}\dfrac{\log\left( \int _{\varOmega(z) } e^{-
			\<\eta^\star(z), \theta-\theta^\star(z)\> }
		d\mu_{\eta^\star(z)}(\theta)\right) }{I(\theta^\star(z))} \leq -1.
	\end{aligned}
	\end{equation}
	Splitting $\theta$ into the normal direction $\hat{n}^\star$
        and orthogonal directions, i.e.,
        $\theta=\theta^\star+s\hat{n}^\star+n^\top, \langle n^\top, \,
        \hat{n}^\star\rangle=0$, we have
        $\<\eta^\star(z),\theta-\thetastar(z)\>=\|\eta^\star(z)\|s$,
        using the fact that $\eta^\star(z)$ is parallel to
        $\hat{n}^\star$ from $\eta^\star$'s definition. In addition,
        $G(z,\infty)=\mu_{\eta^\star(z)}(\varOmega(z))$ and
        $G(z,-\infty)=0$, thus we can use $\partial_sG(z,s) ds$ as a
        new measure.  As in \cite{dematteis2019extreme}, applying
        Fubini's theorem to \eqref{eq:LDT-P(z) tilted measure} using
        the new measure $\partial_sG(z,s) ds$, followed by integration
        by parts, we obtain
	\begin{equation}
	\label{eq:LDT-P(z) wrt G(z,s)}
	\begin{aligned}
	P(z)&
	= e^{-I(\theta^\star(z))}\int_{-\infty}^\infty e^{-\|\eta^\star(z)\| s}
	\partial_sG(z,s) ds
	= e^{-I(\theta^\star(z))}\int_{-\infty}^\infty
	e^{-\|\eta^\star(z)\| s} \|\eta^\star(z)\| G(z,s)\,ds\\
&	\geq e^{-I(\theta^\star(z))}\int_{s_1}^{2s_1}
	e^{-\|\eta^\star(z)\| s} \|\eta^\star(z)\| G(z,s)\,ds 
		\geq e^{-I(\theta^\star(z))} G(z,s_1)\int_{s_1}^{2s_1} \,
	d(- e^{-\|\eta^\star(z)\| s} )\\
&	= e^{-I(\theta^\star(z))} G(z,s_1) e^{-\|\eta^\star(z)\| s_1} (1-e^{-\|\eta^\star(z)\| s_1})
	\geq e^{-I(\thetastar(z))}G(z,s_1) e^{-\|\eta^\star(z)\| s_1} \frac{\|\eta^\star(z)\|s_1}{1+\|\eta^\star(z)\|s_1}.
	\end{aligned}
	\end{equation}
	Applying \Cref{th:as5}, we obtain the lower bound for $P(z)$
	\begin{equation}
	\label{eq:LDT-P(z)lowerbound}
	\lim\limits_{z\to\infty}\dfrac{\log P(z)}{I(\theta^\star(z))}\geq-1+\lim\limits_{z\to\infty}\dfrac{\log G(z,s_1) 
		- \|\eta^\star(z)\| s_1 -\log(1+\|\eta^\star(z)\|^{-1}s_1^{-1})
	}{I(\theta^\star(z))}=-1.
      \end{equation}
      Combining~\eqref{eq:LDT-P(z) up bd}
      and~\eqref{eq:LDT-P(z)lowerbound} establishes~\eqref{eq:LDT}
\end{proof}

\subsection{The LDT optimization problem\label{sec:opt}}

We now discuss the optimization problem \eqref{eq:LDT-instanton},
whose solution is used in \cref{th:1}.
\Cref{th:as1} and \cref{th:as3} imply \eqref{eq:LDT-instanton bdry},
i.e., the minimizer is attained on the boundary of $\varOmega(z)$ and
thus $F(\theta^\star(z))=z$. From the Karush-Kuhn-Tucker (KKT) conditions or
the method of Lagrangian multipliers \cite{borzi2011computational},
and the regularity assumptions in \cref{th:as1}, the minimizer
$\theta^\star(z)$ of \eqref{eq:LDT-instanton bdry} satisfies
\begin{equation}
\label{eq:LDT-KKT}
\nabla I(\theta^\star(z)) = \lambda^\star(z) \nabla
F(\theta^\star(z)),\qquad F(\theta^\star(z))=z,
\end{equation}
where $\lambda^\star(z)\in\RR$ is a Lagrange multiplier.  If $F$ and
$I$ have second derivatives, then the second-order necessary
conditions are:
\begin{equation}
  \label{eq:LDT-SONC}
  \begin{aligned}
    \text{$\forall\theta\in\varOmega$ \ with\ }  \langle \nabla
    I(\theta^\star(z)),(\theta-\theta^\star(z))\rangle = 0: \\
    \left\langle\theta,\left(\nabla^2I(\theta^\star(z))-\lambda^\star(z)\nabla^2F(\theta^\star(z))\right)\theta\right\rangle\ge
    0 .
  \end{aligned}
\end{equation}
That is, the matrix
$\nabla^2I(\theta^\star(z))-\lambda^\star(z)\nabla^2F(\theta^\star(z))$
is positive semidefinite in the tangent space of the constraint. The
sufficient form of this second-order optimality condition, i.e., that the matrix is
positive definite on the tangent space will plan a role in
\cref{sec:PFS}, where we discuss approximations of extreme event
probabilities that rely on the geometry of the extreme event set, and
do not require sampling.

In practice, we are interested in solving the optimization problem
\eqref{eq:LDT-instanton bdry} for different~$z$. This can be done, for
instance, by a projected gradient or Newton descent
method. However, sometimes it is preferable to solve an
unconstrained problem instead of \eqref{eq:LDT-instanton} as discussed
next.

\subsection{Unconstrained formulation of LDT optimization problem}
\label{sec:unconstrained:min}
Here, we study when and in what sense the minimizers of the the
constrained optimization \eqref{eq:LDT-instanton} can also be
found as minimizers of the unconstrained optimization problem
\eqref{eq:LDT-H}, that is,
\begin{equation}
\label{eq:LDT-min}
\begin{aligned}
\min_{\theta\in\varOmega}\ & H(\theta) \:\text{ where }\: H(\theta):=I(\theta)-\lambda F(\theta).
\end{aligned}
\end{equation}
The function $H$ is called the Hamiltonian, e.g., in
\cite{dematteis2019extreme}.  Here, $\lambda>0$ is considered to be a given
constant. If we assume that the problem \eqref{eq:LDT-min} has a unique global
minimizer $\thetastar(\lambda)$ for every fixed $\lambda>0$, then $\thetastar(\lambda)$ is also the global minimizer of \eqref{eq:LDT-instanton bdry} with $z=F(\thetastar(\lambda))$, i.e., of
\begin{equation}
\thetastar(\lambda) =\argmin_{\theta:F(\theta)=F(\thetastar(\lambda))} I(\theta)\,.
\end{equation}
This can be seen as follows: If the minimizer $\thetastar(z)$ of
\eqref{eq:LDT-instanton bdry} with $z=F(\thetastar(\lambda))$ were not
$\thetastar(\lambda)$, from uniqueness of $\thetastar(z)$ in
\Cref{th:as3} we obtain $I(\thetastar(z))< I(\thetastar(\lambda))$
and $F(\thetastar(z))=F(\thetastar(\lambda))$, and thus $H(\thetastar(z))
< H(\thetastar(\lambda))$. This would contradict the assumption that
$\thetastar(\lambda)$ is the unique minimizer of \eqref{eq:LDT-min}.
Thus, under
this assumption, the
minimizer $\thetastar$ of the the LDT problem \eqref{eq:LDT-instanton
  bdry} can also be computed by solving the unconstrained problem
\eqref{eq:LDT-min}.

This provides us with an alternative approach to solve the LDT
optimization problem \eqref{eq:LDT-instanton} for various values
of~$z$. Namely, instead of considering a sequence of~$z$'s in
\eqref{eq:LDT-instanton}, one can consider a sequence of $\lambda$'s
in \eqref{eq:LDT-min}. The solutions $\thetastar(\lambda)$ then
correspond to the extremeness values
$z=z(\lambda):=F(\thetastar(\lambda))$.  Thus, $\lambda>0$ can be used
instead of the threshold $z$ to control the extremeness of the
event. Larger values of $F(\theta)$ correspond to extremer
events. Such events can be found by increasing $\lambda$ which puts
more emphasis on the term involving $F$.  Although the map
$\lambda\to z(\lambda)$ is implicit, solving an unconstrained problem
is often preferable to solving a constrained optimization
problem. This is also the approach we take in
\cref{sec:AP,sec:results}, where we describe our numerical example and
present corresponding results.

In problems where the evaluation of $F$ requires the
solution of a PDE, \eqref{eq:LDT-min} has the typical form
of a PDE-constrained optimization problem, with the analogy that
$I(\theta)$ is a regularization term, and $F(\theta)$
involves the governing PDE.
The existence and uniqueness of solutions for \eqref{eq:LDT-min} 
depend
on properties of
$I(\cdot)$ and $F(\cdot)$, and must be studied on a case-by-case
basis.

\section{Probability estimation using optimization and sampling}\label{sec:PE}
The solutions $\theta^\star(z)$ of \eqref{eq:LDT-min} give the leading
order contributions to the probability, i.e., the log-asymptotic
approximation of $P(z)$ from the large deviation principle
\Cref{th:1}. However, we still lack information regarding the omitted
prefactor $C_0(z)$ in \eqref{eq:PE-prefactor} since LDT only implies
$\log(C_0(z))/I(\theta^\star(z))\to 0$ as $z\to\infty$.  In this
section we explore sampling methods to approximate $C_0(z)$.

\subsection{Conventional Monte Carlo sampling}\label{sec:MC}
Although conventional Monte Carlo sampling is inefficient to study
extreme events, we first summarize its properties to compare
with other methods.
The probability $P(z)$ in \eqref{eq:LDT-probability} can be written as
the expectation of the indicator function for the set
$\varOmega(z)$.
This implies an unbiased estimate of $P(z)$, \cite{liu2008monte},
\begin{equation}
\label{eq:PE-MC estimate P^MC}
P^{MC}_N(z)=\frac{1}{N}\sum_{k=1}^{N}\charOmega{\theta_k},
\end{equation}
where the $\theta_k$'s are i.i.d.\ realizations (samples) from the distribution of $\theta$, i.e., $\theta_k\sim\mu$. 

The mean and the variance of the estimator in \eqref{eq:PE-MC estimate P^MC} are
\begin{equation}
\label{eq:PE-P^MC mean}
\EE_\mu\left[P^{MC}_N(z)\right]=
P(z), \qquad
\VV_\mu\left[P^{MC}_N(z)\right]=
\frac{1}{N}\left[P(z)-P^2(z)\right].
\end{equation}
Thus, the relative root mean square Error (RMSE) is
\begin{equation}
\label{eq:PE-P^MC rel RMSE}
e^{MC}_{N}(z)=\frac{\sqrt{\VV_\mu\left[P^{MC}_N(z)\right]}}{\EE_\mu\left[P^{MC}_N(z)\right]}
=\frac{1}{\sqrt{N}}\frac{\sqrt{P(z)-P^2(z)}}{P(z)}
\approx \frac{1}{\sqrt{N}}\dfrac{1}{\sqrt{P(z)}},
\end{equation}
where the last approximation holds for $z\to\infty$ as $P(z)\ll 1$,
i.e., for extreme events when $P^2(z)$ is dominated by $P(z)$.
Using \eqref{eq:PE-prefactor}, the relative RMSE is
\begin{equation}
\label{eq:PE-P^MC RMSE estimate}
e^{MC}_{N}(z)\approx
\frac{1}{\sqrt{N}}\frac{1}{\sqrt{C_0(z)}}\exp\left(\frac{1}{2}I(\theta^\star(z))\right),
\end{equation}
indicating an exponential term that rapidly increases the number of
samples needed.

For a Gaussian parameter distribution, this term can be computed
explicitly using results detailed in \cref{sec:FORM}. Denoting by
$\theta^\star(z)$ the solution of \eqref{eq:LDT-instanton bdry}, we have
$z=F(\theta^\star(z))$ since the minimizer $\theta^\star(z)$ lies on
the boundary of $\varOmega(z)$. Thus we can use the half-space
approximation \eqref{eq:PE-FORM Gauss} to obtain, for $z\to\infty$
that
\begin{equation}
\label{eq:PE-P(z) Gaussian half-space estimate}
P(z)\approx(2\pi)^{-\frac{1}{2}}\dfrac{\exp(-I(\theta^\star(z)))}{\sqrt{2I(\theta^\star(z))}},
\qquad \text{where} \qquad
I(\theta) = \tfrac12 (\theta-\theta_0) ^T C^{-1}(\theta-\theta_0)
\end{equation}
Hence, the relative RMSE of $P^{MC}_N(z)$ for events with $P(z)\ll 1$
becomes
\begin{equation}
\label{eq:PE-P^MC Gaussian RMSE estimate}
\begin{aligned}
e^{MC}_{N}(z)
\approx\frac{1}{\sqrt{N}}\dfrac{1}{\sqrt{P(z)}}
\approx\frac{1}{\sqrt{N}}\left[ 4\pi I(\theta^\star(z))\right] ^{\frac{1}{4}}\exp\left(\frac{1}{2}I(\theta^\star(z))\right),
\end{aligned}
\end{equation}
where compared to \eqref{eq:PE-P^MC RMSE estimate} we were able to
replace the unknown prefactor with an expression involving the
quadratic rate function $I(\theta^\star(z))$, which satisfies
$I(\theta^\star(z))\to\infty$ as $z\to\infty$ according to
\Cref{th:as3}.

\subsection{Combining Monte Carlo and LDT rate using a constant
  prefactor}\label{sec:fitting}

A simple method to estimate the prefactor $C_0(z)$ is assuming it to
be a constant $C_0$. Although standard MC sampling might not be
effective to study extreme events, it is a reasonable method for
moderately extreme events and can be combined with the rates from LDT
optimization to compute probability estimates for more extreme
events. That is, we determine a constant $C_0$ by fitting
$\exp(-I(\theta^\star(z)))$ to the MC results.  Beside making the
uncontrolled approximation that the prefactor is constant, the method
has another shortcoming: it requires MC sampling to estimate the
probability of moderately extreme events.  In practice, one needs to
choose a regime for fitting, i.e., use the MC estimate for somewhat
extreme events that still have reasonable MC accuracy. Then, LDT can
be used to provide the probability of more extreme events. This
approach was used in \cite{dematteis2019extreme,dematteis2018rogue}.

\subsection{Importance sampling for Gaussian parameters}
From \eqref{eq:PE-P^MC Gaussian RMSE estimate} and
\eqref{eq:PE-P^MC RMSE estimate}, we know that the number of samples
needed for the conventional MC method increases exponentially with
$z$, i.e., as the events become more extreme. For Gaussian parameters,
this can significantly be improved using importance sampling (IS).

For fixed $\lambda>0$, we again denote the solution of
\eqref{eq:LDT-min} by $\theta^\star$, and compute
$z:=F(\theta^\star)$.  The IS method we propose uses a Gaussian
proposal with centered at $\theta^\star$, as sketched in
\cref{fig:IS}.  By inserting $\thetastar-\thetastar$, the probability
$P(z)$ defined in \eqref{eq:LDT-probability} becomes
\begin{equation}
\label{eq:PE-P(z) IS}
\begin{aligned}
P(z)=&(2\pi)^{-n/2}\det(C)^{-1/2}\int_{\varOmega(z)}e^{-\frac{1}{2}\|\theta-\theta^\star+\theta^\star-\theta_0\|^2_{\invC}}d\theta\\
=&e^{-\frac{1}{2}\|\theta^\star-\theta_0\|^2_{\invC}}\cdot(2\pi)^{-n/2}\det(C)^{-1/2}\int_{\varOmega(z)}e^{-(\theta-\theta^\star)^\top \invC(\theta^\star-\theta_0)}e^{-\frac{1}{2}\|\theta-\theta^\star\|^2_{\invC}}d\theta\\
=&e^{-I(\theta^\star)}\EE_{\tilde{\mu}}\left[ \charOmega{\tilde{\theta}}\exp(-(\tilde{\theta}-\theta^\star)^\top \invC(\theta^\star-\theta_0)\right] ,
\end{aligned}
\end{equation}
where $\tilde{\theta}\Gauss{\theta^\star}{C}$ with probability measure
$\tilde{\mu}$.
The corresponding IS estimator is
\begin{equation}
\label{eq:PE-IS estimate P^IS Gaussian}
P^{IS}_N(z)= e^{-I(\theta^\star)}\frac{1}{N}\sum_{k=1}^{N}\left[ \charOmega{\tilde{\theta}_k}\exp(-(\tilde{\theta}_k-\theta^\star)^\top \invC(\theta^\star-\theta_0)\right],
\end{equation}
where $\tilde{\theta}_k$ are independent samples from $\mathcal{N}(\thetastar, C)$.

\begin{figure}[tb]\centering
	\begin{tikzpicture}[]
	\begin{axis}[compat=1.11, width=9cm, height=7cm,
	xmin=-5,
	xmax=4,
	ymin=-3,
	ymax=4,
	axis line style={draw=none},
	tick style={draw=none},
	yticklabels={,,},
	xticklabels={,,},
	]
	\addplot [color=blue!70, only marks, mark=*,mark size=1pt,
	opacity=0.5]
	table[x=x,y=y] {\data/B0sample.txt};
	\addplot [color=green!80!black!70,only marks, mark=*,mark
	size=1pt, opacity=0.5]
	table[x=x,y=y] {\data/Bstarsample.txt};
	\draw [blue,very thick,dotted] (-4,-2) ellipse (2 and 1.41);
	\draw [blue,very thick,dotted] (-4,-2) ellipse (3.46 and 2.45);
	\draw [blue,very thick,dotted] (-4,-2) ellipse (4.9 and 3.46);
	\draw [blue,very thick,dotted] (-4,-2) ellipse (4.9 and 3.46)
	node [xshift=0.6cm, yshift=2.1cm] {level sets of $I(\theta)$};
		\filldraw [fill=red!8,draw=none]
(2.5,-3) .. controls (2,0) and (1,-1)
.. (0,0).. controls (-1,1) and (-0.9,2) .. (-1,4) -- (4,4)  -- (4,-3);
\draw[red!60!black,very thick] (-1,4)  .. controls (-0.9,2)and (-1,1)
.. (0,0).. controls (1,-1) and (2,0) .. (2.5,-3) node[near end, sloped, below]
{$F(\theta)=z$};

	\draw [blue,very thick,dotted] (-4,-2) ellipse (6.93 and 4.9);
	\draw [blue,very thick,dotted] (-4,-2) ellipse (8.49 and 6);
	\draw[black, yshift=4.5cm,xshift=8cm] node {$\varOmega(z)$};
	\draw [green!20!black, very thick, ->] (-3.25,-1.625)	--
	(-1,-0.5) node[above,midway,sloped]{shift};
	\addplot[mark=*] coordinates {(0,0)}  node[yshift=-0.37cm,xshift=-0.35cm]{$\theta^\star(z)$};
	\addplot[mark=*] coordinates {(-4,-2)}  node[yshift=-0.37cm,xshift=-0.35cm]{$\theta_0$};
	\end{axis}
	\end{tikzpicture}
	\caption{Sketch of importance sampling method based on
		shifting the mean $\theta_0$ to the LDT-optimizer
		$\theta^\star$(z) for a specific
		$z$. Samples from the original distribution are
		shown in blue, and those used for IS
		are shown in green.  \label{fig:IS}
	}
\end{figure}
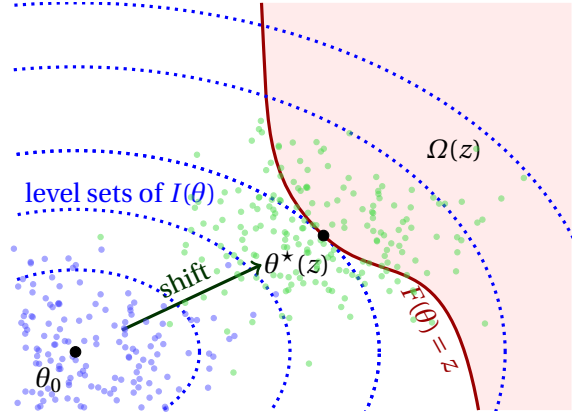

Let us now compute mean, variance and the relative RMSE of this
estimator. Using \eqref{eq:PE-P(z) IS}, the mean and the variance of
the estimator $P^{IS}_N(z)$ are given by
\begin{equation}
\label{eq:PE-IS Gauss mean}
\begin{aligned}
\EE_{\tilde{\mu}}\left[P^{IS}_N(z)\right]
&=e^{-I(\theta^\star)}\EE_{\tilde{\mu}}\left[
\charOmega{\tilde{\theta}}\exp\left(-(\tilde{\theta}-\theta^\star)^\top
\invC(\theta^\star-\theta_0)\right)\right]=P(z), \\
\VV_{\tilde{\mu}}\left[P^{IS}_N(z)\right]
&=e^{-2I(\theta^\star)}\frac{1}{N}\VV_{\tilde{\mu}}\left[ \charOmega{\tilde{\theta}}\exp(-(\tilde{\theta}-\theta^\star)^\top \invC(\theta^\star-\theta_0)\right] .
\end{aligned}
\end{equation}
Since $\theta^\star=\thetastar(z)$ is the solution of
\eqref{eq:LDT-instanton bdry}, \eqref{eq:PE-IS Gauss mean} and the
approximation \eqref{eq:PE-P(z) Gaussian half-space estimate} yields
\begin{equation}
\label{eq:PE-IS Gauss mean estimate}
\EE_{\tilde{\mu}}\left[ \charOmega{\tilde{\theta}}\exp(-(\tilde{\theta}-\theta^\star)^\top \invC(\theta^\star-\theta_0)\right]=e^{I(\theta^\star)}P(z)\approx (2\pi)^{-\frac{1}{2}}\frac{1}{\sqrt{2I(\theta^\star(z))}},
\end{equation}
where
$I(\theta^\star) = \tfrac12(\theta^\star-\theta_0)^\top
\invC(\theta^\star-\theta_0)$.  The sample variance can be estimated
as
\begin{equation}
\label{eq:PE-IS Gauss variance estimate}
\begin{aligned}
&\VV_{\tilde{\mu}}\left[ \charOmega{\tilde{\theta}}\exp(-(\tilde{\theta}-\theta^\star)^\top \invC(\theta^\star-\theta_0)\right]\\
=&\EE_{\tilde{\mu}}\left[\charOmega{\tilde{\theta}}\exp(-2(\tilde{\theta}-\theta^\star)^\top \invC(\theta^\star-\theta_0)\right] -[e^{I(\theta^\star)}P(z)]^2\\
\approx&(2\pi)^{-\frac{1}{2}}\frac{1}{2\sqrt{2I(\theta^\star(z))}}-\left[ (2\pi)^{-\frac{1}{2}}\frac{1}{\sqrt{2I(\theta^\star(z))}}\right]^2
\lesssim(2\pi)^{-\frac{1}{2}}\frac{1}{2\sqrt{2I(\theta^\star(z))}},
\end{aligned}
\end{equation}
where the last estimate holds for $z\to\infty$. Hence, the relative
RMSE  is
\begin{equation}
\label{eq:PE-IS Gauss rel RMSE}
\begin{aligned}
  e^{IS}_N(z)=\dfrac{\sqrt{\VV_{\tilde{\mu}}\left[P^{IS}_N(z)\right]}}
  {\EE_{\tilde{\mu}}\left[P^{IS}_N(z)\right]}
  \approx\frac{1}{\sqrt{N}}\dfrac{\sqrt{(2\pi)^{-\frac{1}{2}}
      \frac{1}{2\sqrt{2I(\theta^\star(z))}}}}{(2\pi)^{-\frac{1}{2}}
    \frac{1}{\sqrt{2I(\theta^\star(z))}}}=\frac{1}{\sqrt{N}}[\pi
  I(\theta^\star(z))]^{\frac{1}{4}}.
\end{aligned}
\end{equation}
Thus, compared to \eqref{eq:PE-P^MC Gaussian RMSE estimate}, we
removed the exponential term of \eqref{eq:PE-P^MC Gaussian RMSE
  estimate} by using importance sampling with samples from
$\mathcal{N}(\theta^\star(z),C)$. This sampling error reduction holds
for all directions. This IS method uses the covariance of the original
distribution in the proposal distribution. Since we know the density
decreases faster in the direction of $\nabla I(\thetastar)$, one may
be able to modify the covariance matrix in this direction in order to
decrease the variance of IS estimator, similar as in the IS method
proposed in \cite{wahalbimc}. Generalizations of the presented
approach to non-Gaussian distributions could rely on approximate
mappings of the parameter distribution to a Gaussian distribution, or
on Gaussian approximations of the distribution about an LDT optimizer.

\section{Probability estimation using second-order approximation of extreme event set}\label{sec:PFS}

Since $P(z)=\mu(\varOmega(z))$, this probability can be computed by
integrating the measure $\mu$ over the set $\varOmega(z)$,
provided we know or can approximate this set. Since evaluation of
$F(\cdot)$ requires the solution of a PDE,
$\varOmega(z)=\{\theta: F(\theta)\geq z\}$ typically cannot be
computed explicitly. However, we can construct an approximation of
$\varOmega(z)$ based on properties of the solution $\theta^\star$ of
\eqref{eq:LDT-min}, and integrate over this approximating set. For
certain distributions, e.g., multivariate Gaussian distributions, this
results in a computationally feasible method. In this section, we
discuss the approximation of $P(z)$ through integration over a
second-order approximation of $\varOmega(z)$, and provide explicit
expressions for multivariate Gaussian parameters. For completeness, we
present corresponding results based on a first-order approximation of
$\varOmega$ in \cref{sec:FORM}. While this first-order approximation
is easier to compute, it is not asymptotically exact in the sense of
\eqref{eq:PE-prefactor}.

For the remainder of this section we consider a Gaussian parameter
distribution \Gaussiancase.  In this case, the LDT minimizer
$\theta^\star$ is also the most probable point, since
$\exp(-I(\theta))$ is the density of the Gaussian distribution up to a
normalization constant; see \cref{ex:gaussian}.  As will be shown in
\cref{sec:SORM}, one can derive explicit approximations of $P(z)$
using approximations of the extreme event set. As preparation step, we
show how to transform the general Gaussian case to a standard normal
distribution $\mathcal{N}(0,I)$.  We also detail how the extreme event
set, the rate function, and the parameter-to-event map are modified
under this transformation.

Although all results in this section are presented in finite
dimensions, we believe that they can be generalized to infinite
dimensions, i.e., Gaussian random fields.  In particular, if the
expressions for the probabilities we find in \cref{thm:SO} converge as
$n\to\infty$, they correspond to probabilities defined over an
infinite-dimensional parameter space. In many cases, such a
convergence follows from properties of the covariance operator of a
Gaussian random field. However, a rigorous discussion of
infinite-dimensional parameter spaces is beyond the scope this present
paper.

We use the optimizer $\theta^\star$ obtained by solving
\eqref{eq:LDT-min} for a fixed $\lambda>0$. The corresponding event
value is $z=F(\theta^\star)$ as discussed in \cref{sec:opt}, i.e.,
$\thetastar = \thetastar(z)$. For simplicity of the notation, we drop
the dependence of $\thetastar$ on $z$ (and $\lambda$) in the
subsequent derivations.
\begin{figure}[tb]\centering
	\begin{tikzpicture}[]
	\begin{axis}[compat=1.11, width=9cm, height=7cm,
	xmin=-5,
	xmax=4,
	ymin=-3,
	ymax=4,
	axis line style={draw=none},
	tick style={draw=none},
	yticklabels={,,},
	xticklabels={,,},
	]
	
	\draw [blue,very thick,dotted] (-4,-2) ellipse (2 and 1.41);
	\draw [blue,very thick,dotted] (-4,-2) ellipse (3.46 and 2.45);
	\draw [blue,very thick,dotted] (-4,-2) ellipse (4.9 and 3.46);
	\draw [blue,very thick,dotted] (-4,-2) ellipse (4.9 and 3.46) node [xshift=0.6cm, yshift=2.1cm] {level sets of $I(\theta)$};
	
		\filldraw [fill=red!8,draw=none]
(2.5,-3) .. controls (2,0) and (1,-1)
.. (0,0).. controls (-1,1) and (-0.9,2) .. (-1,4) -- (4,4)  -- (4,-3);
		\draw[red!60!black,very thick] (-1,4)  .. controls (-0.9,2)and (-1,1)
	.. (0,0).. controls (1,-1) and (2,0) .. (2.5,-3) node[near end, sloped, below]
	{$F(\theta)=z$};

	\draw [blue,very thick,dotted] (-4,-2) ellipse (6.93 and 4.9);
	\draw [blue,very thick,dotted] (-4,-2) ellipse (8.49 and 6);

	\draw [yellow!70!red, very thick] (3,-3) -- (-4,4) node[yshift=-0.6cm,xshift=1.6cm] {\makecell[c]{\textbf{first-order} \\  \textbf{approx.}}};

	\draw [red, line width=0.7mm] 	(4,-1) .. controls (3.95,-0.95) and (1,-1) .. (0,0).. controls (-1,1) and (-0.95,3.95) .. (-1,4) node[yshift=-1cm,xshift=1.5cm] {\makecell[c]{\textbf{second-order} \\  \textbf{approx.}}};
	
	\draw[black, yshift=3cm,xshift=6.3cm] node {\textcolor{red!60!black}{$\varOmega(z)$}}; 		
	
	\filldraw [black] (0,0) circle (2pt) node[yshift=-0.5cm,xshift=0cm]{$\theta^\star$};
	\draw [black, very thick, ->] (0,0)	-- (1,1) node[yshift=-0.3cm,xshift=0.3cm] {$\nstar$};

	\end{axis}
	\end{tikzpicture}
	\caption{2D illustration of the second-order
		approximation of the set $\varOmega(z)$ for given
		$z$. These approximations exploit properties of the
		minimizer $\theta^\star$, the normal direction
		$\nstar:=\nabla_\theta F(\theta^\star)/
		\|\nabla_\theta F(\theta^\star)\|=\nabla_\theta
		I(\theta^\star)/ \|\nabla_\theta I(\theta^\star)\|$
		and the curvature
		of $\partial\varOmega(z)$
		at $\theta^\star$. The first-order approximation is also given, its details are discussed in \cref{sec:FORM}.}\label{fig:FORM-SORM}
\end{figure}
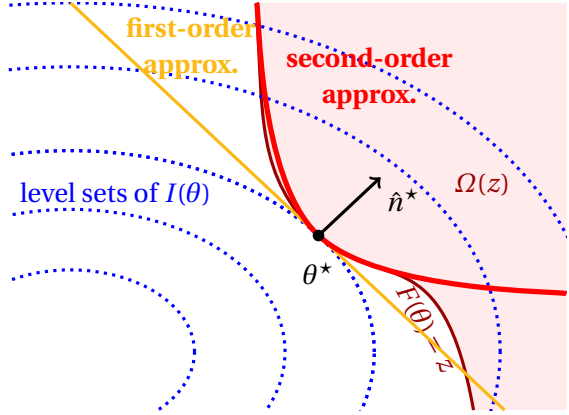
We first define the affine transformation 
\begin{equation}
\label{eq:PE-affine transfer}
\theta=A\xi+\theta_0, \qquad A:=\halfC R,
\end{equation}
where $R$ is a rotation matrix %
such that $R^\top \halfQ(\theta^\star-\theta_0)$ is parallel to the
first unit vector, i.e., only the first component of
$R^\top \halfQ(\theta^\star-\theta_0)$ is nonzero and positive. The
affine transformation \eqref{eq:PE-affine transfer} maps the standard
normal variable $\xi\Gauss{0}{I_n} \text{~with measure~} \musn$ to a
Gaussian variable \Gaussiancase.  Under this transformation, the rate
function and parameter-to-event map become
\begin{equation}
\label{eq:PE-F xi}
\tilde{F}(\xi):=F(\theta)=F(A\xi+\theta_0),\qquad \tilde{I}(\xi):=I(\theta)=I(A\xi+\theta_0)
= \frac{1}{2}\|\xi\|^2.
\end{equation}
The extreme event set $\varOmega(z)$ is mapped to
$\xispace(z)=\{\xi: \ \tilde{F}(\xi)\geq z\}$ and the derivatives
become
\begin{equation}
\label{eq:PE-dF xi}
\begin{array}{ll}
\nabla_{\xi}\tilde{F}(\xi)= A^\top\nabla_\theta F(\theta),&
\nabla^2_{\xi}\tilde{F}(\xi)=A^\top\nabla^2_\theta F(\theta)A,\\
\nabla_{\xi}\tilde{I}(\xi)=A^\top\nabla_\theta I(\theta)=\xi,&
\nabla^2_{\xi}\tilde{I}(\xi)=A^\top\nabla^2_\theta I(\theta)A=I_n.
\end{array}
\end{equation}
The optimizer in the transformed system is
$\xi^\star=\invA(\theta^\star-\theta_0)$ and due to the definition
\eqref{eq:PE-affine transfer}, only the first component of $\xi^\star$
is nonzero
and positive. The following Euler-Lagrange equation holds for
the transformed functions $\tilde{F}$ and $\tilde{I}$:
\begin{equation}
\label{eq:PE-E-L xi}
\dxiI{\xi^\star}=\lambda\dxiF{\xi^\star}.
\end{equation}
By construction, the normal direction at the optimal point $\xistar$ is :
\begin{equation}
\label{eq:PE-n xi}
\dfrac{\nabla_{\xi}\tilde{F}(\xi^\star)}{\|\nabla_{\xi}\tilde{F}(\xi^\star)\|}=\dfrac{\nabla_{\xi}\tilde{I}(\xi^\star)}{\|\nabla_{\xi}\tilde{I}(\xi^\star)\|}=\dfrac{\xi^\star}{\|\xi^\star\|}=\eone,
\end{equation} 
where $\eone$ the first unit vector.
Finally, we introduce
$\Proj:=[\mathbf{0},I_{n-1}]\in\RR^{(n-1)\times n}$, where
$\mathbf{0}\in\RR^{n-1}$ is the zero vector. This matrix represents a
projection onto
$\eospace:=\{\eortho:\langle\eospace,\eone\rangle=0\}=\{[0,\zeta],\zeta\in\RR^{n-1}\}$,
the hyperplane orthogonal to $\eone$.  Clearly,
$\Proj(\eospace) = \RR^{n-1}$ and every vector $\xi$ in $\RR^n$ can be
split uniquely as $\xi = [0,\zeta] + \xi_1e_1$, where
$\zeta=\Proj(\xi)$.

\subsection{Second-Order approximation of $\varOmega(z)$}\label{sec:SORM}

To approximate $\varOmega(z)$, one can use a second-order
approximation of $\partial\varOmega(z)$. This is similar to the
second-order reliability method (SORM) for Gaussian distributions in
engineering \cite{du2001most}, which replaces $F$ in $F(\theta)\geq z$
by its second-order Taylor expansion at $\theta^\star$,
\begin{equation}
\label{eq:PE-SORM F}
\begin{aligned}
F^{SO}(\theta):= F(\theta^\star)+\langle\dthetaF{\theta^\star},\theta-\theta^\star\rangle+\frac{1}{2}\langle\theta-\theta^\star,\ddthetaF{\theta^\star}(\theta-\theta^\star)\rangle.
\end{aligned}
\end{equation}
Since $F(\thetastar)=z$, the corresponding estimate of $P(z)$ becomes
\begin{equation}
\label{eq:PE-SORM}
\begin{aligned}
P^{SO}(z)=\mu\left(\mathcal{Q}(z)\right)=e^{-I(\theta^\star(z))}\int_{-\infty}^\infty
e^{-|\eta^\star(z)| s} |\eta^\star(z)| \mu_{\eta^\star(z)}\left( \mathcal{Q}(z)\backslash\mathcal{H}(z,s)\right) \,ds,
\end{aligned}
\end{equation}
where
\begin{equation}
\label{eq:PE-Q(z)}
\begin{aligned}
\mathcal{Q}(z):=\left\lbrace\theta: F^{SO}(\theta)\geq z \right\rbrace 
=\left\lbrace \theta:\langle\dthetaF{\theta^\star},\theta-\theta^\star \rangle +\frac{1}{2}\langle\theta-\theta^\star, \ddthetaF{\theta^\star}(\theta-\theta^\star)\rangle\geq 0 \right\rbrace.
\end{aligned}
\end{equation}

For a multivariate Gaussian parameter, it is possible to find an
explicit approximation of $P^{SO}(z)$.  First, we start with the
standard normal case.

\begin{lemma}[Second-order approximation for standard normal distribution]\label{lm:SO}
  Let $\xi\Gauss{0}{I_n}$ in\ $\RR^n$ with measure $\musn$,
  $\xi^\star=\xisnorm\eone$, aligned with the first basis vector, is the unique global minimizer of $\|\xi\|^2$
  on the set
  $\Qxi:=\left\lbrace \xi:\left\langle \eone,\xi-\xistar \right\rangle
    +\frac{1}{2}\left\langle \xi-\xistar,H\left( \xi-\xistar\right)
    \right\rangle \geq 0 \right\rbrace $, where $H\in\RR^{n\times n}$
  is a symmetric matrix such that $I_n-\xisnorm H \succeq c_0 I_n$ with
  $c_0 >0$. Then, $\musn(\Qxi)$ satisfies
	\begin{equation}
	\label{eq:PE-SORM Gauss xi}
	\begin{aligned}
	\musn(\Qxi)
	\approx(2\pi)^{-\frac{1}{2}}\frac{1}{\xisnorm}e^{-\frac{1}{2}\|\xistar\|^2}\prod_{i=1}^{n-1}\left[1- \xisnorm\lambda_i\left( H_1 \right)  \right]^{-\frac{1}{2}},
	\end{aligned}
	\end{equation}
	where the asymptotic estimate holds for
	$\xisnorm\to\infty$. Here, $H_1:=\Proj
	H\Proj^\top\in\RR^{(n-1)\times(n-1)}$ is the submatrix
	obtained by removing the first row and column of $H$,
	and $\lambda_i(\cdot)$ denotes the $i$-th eigenvalue.
\end{lemma}
\begin{proof}
	
  First, we split $\xi$ as
  $\xi=\xistar+\|\xistar\| (s\eone+\eortho),\ s\in\RR,\ \eortho\in\eospace$, and use  the property that $\eone$ and $\eortho$
  are orthogonal to
  obtain
\begin{equation*}
\begin{aligned}
\Qxi & =\left\lbrace\xi: \left\langle \eone,\xi-\xistar \right\rangle +\frac{1}{2}\left\langle \xi-\xistar,H\left( \xi-\xistar\right) \right\rangle\geq 0 \right\rbrace \\
& =\left\lbrace\xi: s +\frac{1}{2}\left\langle  s\eone+\eortho,\|\xistar\|H\left(  s\eone+\eortho\right) \right\rangle\geq 0 \right\rbrace \subseteq \left\lbrace\xi: s +\frac{1}{2}(1-c_0) \left(s^2+ \|\zeta\|_{\RR^{n-1}}^2\right)\geq 0 \right\rbrace, \\
\end{aligned}
\end{equation*}
	where $\zeta:=\Proj(\eortho)\in\RR^{n-1}$ and the last relation is obtained by applying $I_n-\xisnorm H \succeq c_0 I_n$ to $s\eone+\eortho$. Thus, for $\forall \xi \in \Qxi$,
		\begin{equation}
	\label{eq:PE-xinorm-asymp}
	\xinorm^2=\|\xistar\|^2\left( 1+2s+s^2+\|\zeta\|_{\RR^{n-1}}^2\right) \ge \xisnorm^2 \left( 1+c_0 s^2 + c_0\|\zeta\|^2_{\RR^{n-1}} \right).
	\end{equation}
	Since this term is in the exponent of the probability density $e^{-\frac{1}{2}\xinorm^2}$, the mass will be concentrated around the part that $|s|$ and $\|\zeta\|_{\RR^{n-1}} $ are close to zero as $\|\xistar\|\to \infty$. Further, from the equality in $\eqref{eq:PE-xinorm-asymp}$, we conclude that $\|\zeta\|_{\RR^{n-1}} = O(1/\|\xistar\|)$ and $s = O(1/\|\xistar\|^2)$. In this regime, the exponent of the integrand becomes
		\begin{equation}
	\label{eq:PE-xinorm-SORM}
	\xinorm^2=\|\xistar\|^2\left( 1+2s+s^2+\|\zeta\|_{\RR^{n-1}}^2\right) \approx \|\xistar\|^2\left( 1+2s+\|\zeta\|_{\RR^{n-1}}^2\right),
	\end{equation}
       and the term in the definition of
        $\Qxi$ becomes
	\begin{equation}
	\label{eq:PE-SORM-Qxi-split}
	\begin{aligned}
	&\left\langle \eone,\xi-\xistar \right\rangle +\frac{1}{2}\left\langle \xi-\xistar,H\left( \xi-\xistar\right) \right\rangle\\
	& =\xisnorm \left( s+\frac{\xisnorm H_{11}}{2} s^2+ \langle \xisnorm H_{2\cdots
          n,1},\zeta\rangle_{\RR^{n-1}}s + \frac{1}{2}\langle\zeta,\xisnorm H_1\zeta\rangle_{\RR^{n-1}}\right)\\
      &\approx \xisnorm \left( s+ \frac{1}{2}\langle\zeta,\xisnorm H_1\zeta\rangle_{\RR^{n-1}}\right)
       ,
	\end{aligned}
	\end{equation}
	where $H_{11}\in\RR$ is the $(1,1)$-entry of $H$, and
        $H_{2\cdots n,1}\in\RR^{n-1}$ is the first column of $H$
        without the first component. Here, we use the asymptotics of $s$ and $\zeta$ when $\xisnorm\to \infty$ and also think $\xisnorm H = O(1)$ from the condition $I_n-\xisnorm H \succeq c_0 I_n$. Thus, we can compute the measure
        $\musn(\Qxi)$ using integration over $\eone$ and its
        orthogonal complement and the asymptotic estimates  \eqref{eq:PE-xinorm-SORM}, 
        \eqref{eq:PE-SORM-Qxi-split} and Fubini's
        theorem
         to obtain the asymptotic estimate
	\begin{align*} 
		\musn(\Qxi)
	&=(2\pi)^{-\frac{n}{2}}\int_{ s+\frac{\xisnorm H_{11}}{2} s^2+ \langle \xisnorm H_{2\cdots
			n,1},\zeta\rangle_{\RR^{n-1}}s + \frac{1}{2}\langle\zeta,\xisnorm H_1\zeta\rangle_{\RR^{n-1}}\geq 0}\hspace*{-9em} e^{-\frac{1}{2}\|\xistar\|^2\left( 1+2s+s^2+\|\zeta\|_{\RR^{n-1}}^2\right)} \xisnorm^n d \zeta ds\\
&\approx(2\pi)^{-\frac{n}{2}}\int_{s+ \frac{1}{2}\langle\zeta,\xisnorm H_1\zeta\rangle_{\RR^{n-1}} \geq
		0}e^{-\frac{1}{2}\|\xistar\|^2\left( 1+2s+\|\zeta\|_{\RR^{n-1}}^2\right)} \xisnorm^n d \zeta
	ds\\
	&=(2\pi)^{-\frac{n}{2}} \xisnorm^n e^{-\frac{1}{2}\|\xistar\|^2}  \int_{\RR^{n-1}}e^{-\frac{1}{2}\xisnorm^2\|\zeta\|_{\RR^{n-1}}^2}\left(\int_{-\frac{1}{2}\langle\zeta,\xisnorm H_1\zeta \rangle_{\RR^{n-1}} }^{\infty}e^{-\|\xi^\star\|^2s}   ds \right)d\zeta\\
	&=(2\pi)^{-\frac{n}{2}} \xisnorm^{n-2} e^{-\frac{1}{2}\|\xistar\|^2}\int_{\RR^{n-1}} e^{-\frac{1}{2}\left\langle \zeta,\xisnorm^2\left(I_{n-1}-\xisnorm H_1 \right)\zeta \right\rangle_{\RR^{n-1}}  } d\zeta.\\
	\intertext{The assumption $I_n-\xisnorm H \succeq c_0 I_n$ with $c_0>0$ implies that $I_{n-1}-\xisnorm H_1$ is
          positive, and thus we obtain that}
      \musn(\Qxi) &\approx
	(2\pi)^{-\frac{1}{2}} \xisnorm^{n-2} e^{-\frac{1}{2}\|\xistar\|^2}\det\left[ \xisnorm^2\left(I_{n-1}-\xisnorm H_1 \right)\right]^{-\frac{1}{2}}\\
	&=(2\pi)^{-\frac{1}{2}} \dfrac{1}{\|\xistar\|}e^{-\frac{1}{2}\|\xistar\|^2}\prod_{i=1}^{n-1}\left[1 -\xisnorm\cdot \lambda_i\left( H_1\right)  \right]^{-\frac{1}{2}}.
	\end{align*}
\end{proof}
Note that the condition $\<\xi,(I-\xisnorm H)\xi\> >0$ in \cref{lm:SO}
is equivalent to $1/\xisnorm >
\<\xi,H\xi\>/\|\xi\|^2
$. Geometrically, this condition means that
the curvature of the centered circle through $\xistar$ must be an upper bound
for the eigenvalues of $H_1$, i.e., the projection of $H$ onto the
space orthogonal to $e_1$. This circle is the level set through
$\xistar$ of the rate function for the standard normal distribution.
In the generalization of \cref{lm:SO} presented next, such a condition
follows from the second-order optimality condition of the
LDT-minimizer. This result uses the affine transformation
\eqref{eq:PE-affine transfer} and applies \cref{lm:SO} to obtain an
approximation of $P^{SO}(z)$ in \eqref{eq:PE-SORM}.
\begin{theorem}[Second-order approximation for general Gaussian
  distributions]\label{thm:SO}
  Let \Gaussiancase, denote as $\thetastar(z)$ the optimizer with
  $\lambda>0$ of \eqref{eq:LDT-min}, and define the rotation operator
  as in \eqref{eq:PE-affine transfer}. Additionally, assume $F$ is
  twice continuously differentiable and $\thetastar$ is also the unique global minimizer of $I(\cdot ) $on $\mathcal{Q}(z)$,
  satisfying that $I_n -\lambda (\halfC)^\top \nabla^2_\theta F(\thetastar) \halfC\succeq c_0 I_n$ with $c_0>0$. Then, the second-order
  approximation $P^{SO}(z)$ defined in \eqref{eq:PE-SORM} can be
  approximated as
	\begin{equation}
	\label{eq:PE-SORM Gauss}
	\begin{aligned}
	P^{SO}(z)\approx \dfrac{(2\pi)^{-1/2}}{\sqrt{2I(\theta^\star(z))}}e^{-I(\theta^\star(z))}\prod_{i=1}^{n-1}\left[1 -\lambda
	\lambda_i\left( \Proj R^\top(\halfC)^\top\ddthetaF{\theta^\star(z)} \halfC R \Proj^\top \right)  \right]^{-\frac{1}{2}},
	\end{aligned}
	\end{equation}
	where the asymptotic estimate holds for $z\to\infty$. As before,
	$\lambda_i(\cdot)$ is the $i$-th eigenvalue and $\Proj$ is the
	projection onto the subspace orthogonal to the first basis vector.
\end{theorem}
\begin{proof}
	Using \eqref{eq:PE-dF xi}, the set $\mathcal{Q}(z)$ defined in
        \eqref{eq:PE-Q(z)} is affinely transformed to
	\begin{equation*}
	\begin{aligned}
	&	\left\lbrace \xi: \left\langle A^{-\top} \dxiF{\xistar} ,A\xi-A\xistar \right\rangle  +\frac{1}{2}\left\langle A\xi-A\xistar, A^{-\top} \ddxiF{\xistar} A^{-1} (A\xi-A\xistar)\right\rangle \geq 0   \right\rbrace \\
	=&	\left\lbrace \xi: \|\dxiF{\xistar}\|  \left\langle\eone,\xi-\xistar \right\rangle  +\frac{1}{2}\left\langle \xi-\xistar, \ddxiF{\xistar} (\xi-\xistar)\right\rangle \geq 0  \right\rbrace \\
	=&\left\lbrace \xi:\left\langle\eone,\xi-\xistar\right\rangle  +\frac{1}{2}\left\langle \xi-\xistar, H (\xi-\xistar)\right\rangle   \right\rbrace=\Qxi,
	\end{aligned}
	\end{equation*}
	with $H=\ddxiF{\xistar}/ \|\dxiF{\xistar}\| $. Thus,
	$P^{SO}(z)=\mu(\mathcal{Q}(z))=\musn(\Qxi)$, for which we use
	\cref{lm:SO}. Combining \eqref{eq:PE-dF xi} and the
	Euler-Lagrange equation \eqref{eq:PE-E-L xi}, we have
	\begin{equation*}
	H=\frac{\ddxiF{\xistar} }{\|\dxiF{\xistar}\|}=\frac{{\|\dxiI{\xistar}\|}}{\|\dxiF{\xistar}\|}
	\frac{A^\top \ddthetaF{\thetastar} A}{\xisnorm}.
	\end{equation*}
	Thus, $I_n-\xisnorm H = I_n - \lambda A^\top \ddthetaF{\thetastar} A = R^\top (I_n - \lambda (\halfC)^\top \nabla^2_\theta F(\thetastar) \halfC) R \succeq c_0 I_n $, satisfying the assumption in \cref{lm:SO}.
	Using this $H$ and $\xisnorm=\sqrt{2I(\thetastar)}$ from
        \eqref{eq:PE-F xi} in \cref{lm:SO}, we obtain 
	\begin{equation*}
	P^{SO}(z)\approx \dfrac{(2\pi)^{-1/2}}{\sqrt{2I(\theta^\star(z))}}e^{-I(\theta^\star(z))}\prod_{i=1}^{n-1}\left[1 -\lambda
	\lambda_i\left( \Proj A^\top\ddthetaF{\theta^\star} A\Proj^\top \right)  \right]^{-\frac{1}{2}}.
	\end{equation*}
	Using the definition of the linear operator $A$ in
	\eqref{eq:PE-affine transfer} finishes the proof.
\end{proof}
Note that \eqref{eq:PE-SORM Gauss} also holds when $P^{SO}(z)$ is
replaced by $P(z)$, if we further assume $\{\xi: F(A\xi+\theta_0) \geq z \}\subseteq \{ \xi: \hat{F}(\xi/r(z)) \geq 0 \}$, where $r(z)$ is a monotonically increasing function and $r(z)\to \infty$ as $z\to \infty$, $\xistar(z)=A^{-1}(\thetastar(z) -\theta_0)$ is the unique global minimizer of $\xinorm^2$ on $\{ \xi: \hat{F}(\xi/r(z)) \geq 0 \}$ and  $I_n - \nabla^2\hat{F}(\xistar(z)) \succ 0$. This follows from asymptotic expansions of
multi-normal Laplace-type integrals
\cite[Chapter~8]{bleistein1986asymptotic} and
\cite[Appendix~I]{breitung1984asymptotic} using that $F$ is twice
differentiable, that $\thetastar(z)$ is the minimizer of $I(\theta)$
over $\varOmega(z)$, and that $I(\thetastar(z))\to \infty$ as
$z\to \infty$ as assumed in \Cref{th:as3}.  Thus, $P^{SO}(z)$ is an
asymptotic approximation of $P(z)$ and we obtain an asymptotic
approximation of the prefactor $C_0(z)$, i.e.,
$P(z)\approx C_0(z)\exp(-I(\theta^\star(z))), \text{as } z\to\infty$,
where $C_0(z)$ is given by the right hand side in \eqref{eq:PE-SORM
  Gauss} neglecting the exponential term.

Compared to \eqref{eq:PE-SORM Gauss}, the probability estimation based
on the first-order (other than the second-order) approximation of
$\varOmega(z)$ is easier to compute.  This approach, which is known in
engineering as First-Order Reliability Method (FORM) is summarized in
\cref{sec:FORM}. While it only requires $\thetastar$, it does not
provide a controllable approximation of the prefactor $C_0(z)$. In
fact, FORM must be multiplied with a correction factor to obtain an
asymptotically exact approximation. This leads to an alternative
approach to approximate $P^{SO}(z)$ typically used in
engineering. Namely, using the Euler-Lagrange equations
\eqref{eq:PE-E-L xi} and the first-order approximation
\eqref{eq:PE-FORM Gauss}, we can reinterpret \eqref{eq:PE-SORM Gauss}
as a refinement of $P^{FO}(z)$ with a correction term:
\begin{equation}
\label{eq:PE=SORM k_i}
P^{SO}(z)\approx P^{FO}(z)\prod_{i=1}^{n-1}\left(1+\sqrt{2I(\theta^\star(z))} k_i\right)^{-1/2}.
\end{equation}
Here, the $k_i$'s are the eigenvalues of
$-\Proj A^\top\ddthetaF{\theta^\star}
A\Proj^\top/\|A^\top\dthetaF{\theta^\star}\|$, i.e., the principle
curvatures of $F$ at $\theta^\star$.  This is the formulation that is
referred to as SORM in engineering, where the curvatures $k_i$ are
typically computed directly as detailed in \cite{du2001most}. However,
we prefer the formulation \eqref{eq:PE-SORM Gauss} over
\eqref{eq:PE=SORM k_i} as it lends itself to approximating dominating
eigenvalues with low-rank methods, which is particularly useful for
high parameter dimensions. This approach, which to the best of our
knowledge is novel, is presented next.

\subsection{Low-rank approximation of covariance-preconditioned
	Hessian of $F$} \label{subsec:low-rankSORM}
A natural question is if the approximation for $P^{SO}(z)$ presented
in \cref{thm:SO} can be computed efficiently. In particular for
problems where the parameter dimension $n$ is large, and where the
definition of $F$ involves the solution of an expensive-to-solve PDE,
computation of the Hessian matrix $\ddthetaF{\theta^\star}$ may be
infeasible as computation of each of its columns requires at least two
PDE solves. However, \eqref{eq:PE-SORM Gauss} shows that mostly the
eigenvalues of $\Proj R^\top(\halfC)^\top\ddthetaF{\theta^\star}
\halfC R \Proj^\top$ that are significantly different from zero
contribute to the product in \eqref{eq:PE-SORM Gauss} and thus to the
estimate for $P^{SO}(z)$. Geometrically, these eigenvalues correspond
to directions in which the boundary $\partial\varOmega(z)$ has
large curvature. Additionally, these directions must correspond to
large eigenvalues of the covariance matrix $C$, i.e., they must also be
important for the underlying Gaussian distribution.

Using either the Lanczos algorithm or a randomized SVD
\cite{HalkoMartinssonTropp11, CullumWilloughby85} allows to compute
the dominant eigenvalues of
$\Proj R^\top(\halfC)^\top\ddthetaF{\theta^\star} \halfC R \Proj^\top$
without explicit construction of this matrix but only through
application to vectors. The number of required matrix-vector
applications for these methods is typically only slightly larger than
the number of dominant eigenvalues. This number depends on properties
of $\ddthetaF{\theta^\star}$ and $C$. While one cannot make general
statements about the number of dominant eigenvalues, we show in
\cref{sec:low-rank} that for our tsunami example, this number is
small, and is insensitive to $\lambda>0$. Such a low-rank property is
likely to also hold for other problems due to the structure of the
matrix $(\halfC)^\top\ddthetaF{\theta^\star} \halfC$, which we refer
to as \emph{covariance-preconditioned parameter-to-event Hessian}. A
similar operator occurs in Bayesian inverse problems, where it is
referred to as the prior-preconditioned misfit Hessian
\cite{Bui-ThanhGhattasMartinEtAl13}.  Dominant eigenvalues of
$\ddthetaF{\theta^\star}$ correspond to directions with strong (either
positive or negative) curvature of $\partial\varOmega(z)$, i.e., their
occurrence depends on the nonlinearity of the parameter-to-event
map. Large eigenvalues of $C$ correspond to directions with large
variance, i.e., where the Gaussian measure has the majority of its
mass. Only parameter directions that are important for
$\ddthetaF{\theta^\star}$ and for $C$ have eigenvalues with a large
absolute value and thus contribute significantly to the right hand
side in \eqref{eq:PE-SORM Gauss}.

\section{Application to extreme tsunami probability estimation}\label{sec:AP}
As our main application, we study earthquake-induced tsunamis and
estimate the probability that they give rise to an extreme flooding
event on shore.  Tsunamis are caused by a sudden elevation change of
the ocean floor after fast, and potentially complex, slip at the fault
between two tectonic plates below the ocean floor. This slip process,
also called dynamic rupture, is caused by stress buildup over years or
decades. It typically occurs within seconds or, for the largest events
a few minutes.  In particular for large events, slip patterns are
complex and difficult to predict.  Hence, we model sudden ocean floor
elevation changes as a random parameter field. Since the fault slip
process is on a much faster time scale than the scale at which water
waves travel, we do not include time dependence in this random process
and consider the ocean floor elevation change as instantaneous.  The
map from these (random) parameters to the event, namely the average
wave height in a region close to shore, is governed by the shallow
water equation. Here, for simplicity, we use a one-dimensional shallow
water model. The next subsections describe the shallow water equations
and their discretization, modeling the distribution of the parameter
field, the parameter-to-event map and the computation of its
derivatives.  Numerical results in which we study the performance of
the proposed methods and the physics implications are presented in
\cref{sec:results}.

\subsection{One-dimensional shallow water equations}
To model tsunami waves, we use the one-dimensional shallow water
equations \cite{leveque2008high} defined on a domain $\mathcal{D}=[a,b]$
for times $t\in[0,T_F]$. The domain represents a slice through the sea,
that includes the shallow part near the shore and the part
where the ocean floor elevation can change.
We denote the horizontal fluid velocity as $u(x,t)$ and the height of
water above the ocean floor by $h(x,t)$. The bathymetry $B(x)$ is the
negative depth of the ocean at rest, i.e., $h(x,t)+B(x)=0$ when the
ocean is at rest.  The shallow water equations in conservative form
are
\begin{equation}
\label{eq:AP-SWE}
\left[ \begin{array}{c}
h\\ 
hu
\end{array} \right]_t+\left[ \begin{array}{c}
hu\\ 
hu^2+\frac{1}{2}gh^2
\end{array} \right]_x=\left[ \begin{array}{c}
0\\ 
-ghB_x
\end{array} \right],
\end{equation}
where $g$ is the gravitational constant and the subscripts $t,x$ denote
derivatives with respect to time and location.
Introducing the variable $v:=hu$ and augmenting \eqref{eq:AP-SWE} with
initial and boundary condition leads to
\begin{subequations}
\label{eq:AP-PDE}
\begin{alignat}{2}
  h_t+v_x&=0 \quad && \text{ on } \mathcal{D}\times [0,T_F],\label{eq:AP-PDE1}\\
  v_t+\left(\frac{v^2}{h}+\frac{1}{2}gh^2\right)_x+ghB_x&=0 && \text{ on
  } \mathcal{D}\times [0,T_F],\label{eq:AP-PDE2}\\
h(x,0)=-B_0(x),\  v(x,0)&=0 && \text{ for } x\in \mathcal{D}, \label{eq:AP-PDE3}\\
v(a,t)=v(b,t)&=0  && \text{ for } t\in [0,T_F].\label{eq:AP-PDE4}
\end{alignat}
\end{subequations}
Here, the initial condition \eqref{eq:AP-PDE3} assumes that the water
is at rest. It can be verified that if $B=B_0$, $h=-B_0$ and $v=hu=0$ for all
times. However, any change in the bathymetry $B$ results in a nonzero
solution. This is the main mechanism that generates tsunami
waves. Note that this form of the shallow water equations only allows
to incorporate the vertical bathymetry change $B-B_0$. Earthquakes
also alter the horizontal component of the bathymetry, but most likely
this does not have a large effect on tsunami waves. The reflective
boundary conditions \eqref{eq:AP-PDE4} are are not physically
accurate, but we assume that the boundary is far enough from the
region where the tsunami wave is generated or measured such that
unphysical reflections are not relevant. For a discussion on different
boundary conditions for the shallow water equations, we refer to
\cite{vreugdenhil2013numerical}.

The domain we use for our tsunami model problem is shown in
\cref{fig:tohoprobset}. This setup is inspired by the 2011 Tohoku-Oki
earthquake and tsunami \cite{FujiwaraKodairaNoEtAl11}.  The geometry
represents a two-dimensional slice with a bathymetry that models the
continental shelf and the pacific ocean to the east of Japan. We also
use a similar slip mechanism as occurred in the Tohoku-Oki earthquake,
as discussed next.

\begin{figure}[tb]
	\centering
	\begin{tikzpicture}[]
	\begin{axis}[compat=1.10, width=12.5cm, height=3.2cm, scale only axis,
	xmin=0, xmax=400, ymin=-11,ymax=.1,
	xlabel={Distance to shore [km]},
	ylabel={Depth [km]},
         y tick label style={
           /pgf/number format/.cd,
           scaled y ticks = false,
           set thousands separator={,},
           fixed,
         },
         ytick = {0,-2,-4,-6,-8,-10},
	legend style={font=\tiny,nodes=right}, 
	legend pos=south west]
\path[name path=sealevel] (axis cs: 0,0) -- (axis cs: 400,0);
	\addplot[name path=bathymetry,,orange!70,very
          thick,domain=0:400,dashed, samples=200] {1*(-0.05*(x<=45)+(-0.05-3.95/110*(x-45))*(x>45)*(x<=155)+
		(-4-4/45*(x-155))*(x>155)*(x<=200)+
		(-8+4/100*(x-200))*(x>200)*(x<=300)-
		4*(x>300))};	
	\addlegendentry{Reference bathymetry}
	\addplot[blue!75!green!30!white] fill between[ 
	of = sealevel and bathymetry, 
	soft clip={domain=0:400},
	];	\addlegendentry{Water}
	\path[name path=bottom] (axis cs: 0,-11) -- (axis cs: 400,-11);
	\addplot[orange!30!black!10!white] fill between[ 
	of =  bathymetry and bottom, 
	soft clip={domain=0:400},
	];	\addlegendentry{Rock/Sand}
	\addplot[red,line width=3pt,domain=40:44]
                {0};\addlegendentry{Observation location}
		\addplot[purple!60!blue,very thick,domain=165:195] {1*(-0.05*(x<=45)+(-0.05-3.95/110*(x-45))*(x>45)*(x<=155)+
			(-4-4/45*(x-155))*(x>155)*(x<=200)+
			(-8+4/100*(x-200))*(x>200)*(x<=300)-
			4*(x>300))};	\addlegendentry{Main up/downlift}
		\addplot [color=green!70!red, mark=none, mark size=2.5pt, very thick]
	table[x expr=\thisrowno{0}*0.001,y expr=\thisrowno{1}*0.001]
        {\data/fault.txt};\addlegendentry{Slip location}
	\end{axis}
	\end{tikzpicture}
	\caption{Problem setup inspired by Tohoku-Oki 2011
          earthquake/tsunami. Bathymetry changes (area in purple) are
          modeled as resulting from 20 randomly slipping patches in
          the slip region (in green, with end points $(178km,-9.9km)$ and
          $(187km,-9.1km)$) using the Okada model.
          The event we observe is the average wave height
          in the interval [40km,44km] close to shore (shown in red),
          where the water depth at rest is 50m.
        } \label{fig:tohoprobset}
\end{figure}
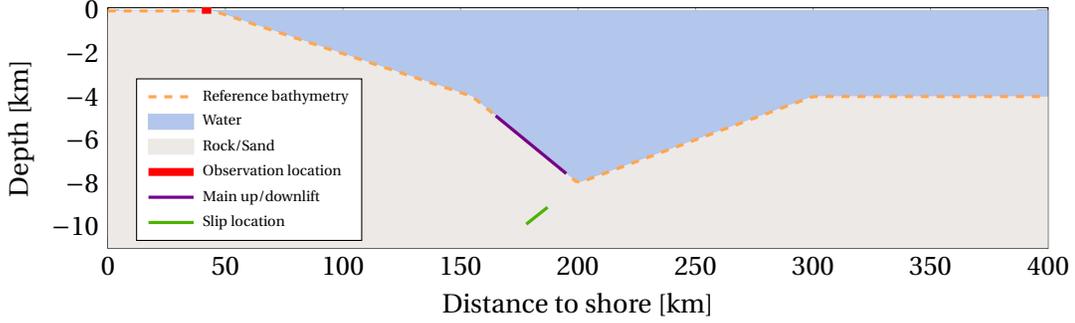

\subsection{Modeling random parameter field $B$ using subduction
  physics}\label{subsec:tsunamiB} The bathymetry $B(x)$, whose derivative
enters in the right hand side of \eqref{eq:AP-PDE}, changes during an
earthquake as a result of slip between plates under the ocean
floor. Since details of this slip process are difficult to predict, we
model the slip as a random process, and thus also the bathymetry field
$B$ is random. Since $B$ enters in the shallow water equations
\eqref{eq:AP-PDE}, the (space and time-dependent) solutions $h$ and
$v$ are random and hence also the event objective we will specify in
\cref{subsec:tsunamiF} is a random variable.

The relation between slip under the sea floor and the resulting
bathymetry change typically assumes that the earth's solid crust
behaves like a linear elastic material.  The commonly used Okada model
\cite{Okada85} assumes a finite number of slip patches in a fault
under the ocean floor, and evaluates expressions for a linear elastic
material to compute the induced bathymetry change.  We assume 20 slip
patches and model each of the uncertain slips of fault pairs as
independent Gaussian random parameter with mean zero and a standard
deviation of 10m. We assume the slip to be along the down-dip
direction, i.e., a positive slip value means that the overriding plate
(i.e., the sea floor) moves downwards along the fault while a negative
value means it is moving upwards. In this work we use a centered
Gaussian slip distribution, which is a simplification as realistic
earthquake slips are typically negative since they are caused by a sudden
stress release. 
We refer to \cite{leveque2016generating,gao2018defining} for more realistic
slip distribution models, which we are currently incorporating into
our framework. 
The
Okada model is defined for three-dimensional sea floor
deformations. By assuming that the width of each patch is infinite and
extracting the deformation in the direction of the slice plane, we
adopt the Okada implementation \cite{OkadaMatlab} to our
two-dimensional geometry. The model assumes that the crust has a
Poisson's ratio of $\nu=0.25$, which is the only elasticity parameter
that plays a role in the Okada model. The linear relationship between
skip patches and bathymetry change results in
\begin{equation}\label{eq:Okada->B}
  B(x) =  B_0(x) + (O S)(x) \: \text{ with } \:S=(s_1,\ldots,s_{20})^\top \:\text{ and }\: (OS)(x) := \sum_{i=1}^{20}s_i O_i(x),
\end{equation}
where $O_i$ is the bathymetry change due to the $i$-th slip patch, and
$s_i\sim \mathcal N(0,10)$. Hence
\begin{equation}
B\in \mathcal{B} :=\left\lbrace B_0(x) + \sum_{i=1}^{20}s_i O_i(x): s_i\in
\RR\right\rbrace.
\end{equation}
Random draws of the bathymetry change
$B-B_0$ are shown in \cref{fig:Bsamples}. While the slips are
independent, the bathymetry samples are smooth. This is due to
properties of linear elasticity, i.e., rough boundary conditions on
one part of the boundary result in a smooth displacement field on a
different part of the boundary. Note also that all random samples of
$B-B_0$ yield positive \emph{and} negative elevation changes as
typically also found in observations
\cite{FujiwaraKodairaNoEtAl11}. This is due to the fact that slip at
the fault zone is tangential and thus leads to elastic compression in
parts of the elastic domain and to extension in other parts.

Since the transformation \eqref{eq:Okada->B} between slips and the
bathymetry change is linear, $B$ is a Gaussian random field with mean
$B_0$ and covariance induced by the slip covariance matrix $C_s:=100
I_{20}$.  The rate function $I$ for a bathymetry $B\in \mathcal B$
with coefficient vector $S\in \RR^{20}$ is
\begin{equation}
\label{eq:AP-I}
I(B)=\frac 12 S^\top C_s^{-1} S =:\frac{1}{2}\langle\!\!\langle B-B_0,B-B_0\rangle\!\!\rangle_{C_s}.
\end{equation}

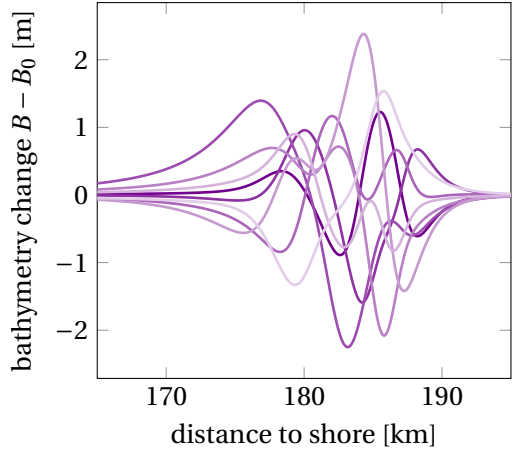
\begin{figure}[tb]\centering
	\begin{tikzpicture}[]
	\begin{axis}[compat=newest, width=5.5cm, height=5cm, scale only axis,
	xlabel={distance to shore [km]},
	ylabel={bathymetry change $B-B_0$ [m]},
	legend style={font=\tiny,nodes=right}, xmin=165, xmax=195]
	\addplot [color=purple!60!blue, line width=1]
	table[x expr=\thisrowno{0}*0.001,y=B] {\data/Toho_samples1std10nos20.txt};
	\addplot [color=purple!60!blue!80!white, mark=none, mark size=2pt, line width=1]
	table[x expr=\thisrowno{0}*0.001,y=B] {\data/Toho_samples2std10nos20.txt};
	\addplot [color=purple!60!blue!70!white, mark=none, mark size=2pt, line width=1]
	table[x expr=\thisrowno{0}*0.001,y=B] {\data/Toho_samples3std10nos20.txt};
	\addplot [color=purple!60!blue!60!white, mark=none, mark size=2pt, line width=1]
	table[x expr=\thisrowno{0}*0.001,y=B] {\data/Toho_samples4std10nos20.txt};
	\addplot [color=purple!60!blue!50!white, mark=none, mark size=2pt, line
	width=1]
	table[x expr=\thisrowno{0}*0.001,y=B] {\data/Toho_samples5std10nos20.txt};
	\addplot [color=purple!60!blue!40!white, mark=none, mark size=2pt, line
	width=1]
	table[x expr=\thisrowno{0}*0.001,y=B] {\data/Toho_samples6std10nos20.txt};
	\addplot [color=purple!60!blue!30!white, mark=none, mark size=2pt, line
	width=1]
	table[x expr=\thisrowno{0}*0.001,y=B] {\data/Toho_samples7std10nos20.txt};
	\addplot [color=purple!60!blue!20!white, mark=none, mark size=2pt, line
	width=1]
	table[x expr=\thisrowno{0}*0.001,y=B] {\data/Toho_samples8std10nos20.txt};
	\end{axis}
	\end{tikzpicture}
	\caption{Samples from the bathymetry change distribution
          computed from the Okada model with 20 slip fault pairs under
          the ocean floor.  Shown is the vertical ocean floor
          displacement.  Each slip is independent with mean zero and
          standard deviation of 10m. The main part of the ocean floor
          where bathymetry change arises is highlighted in purple in
          \Cref{fig:tohoprobset}.}\label{fig:Bsamples}
\end{figure}

\subsection{Measuring tsunami size close to shore}\label{subsec:tsunamiF}
After discussing the governing equations and the parameter
distribution for $B$, it remains to define how we measure events.
Namely, to measure the size of a tsunami close to shore, we average the wave height
$(h+B_0)$ in the area $[c,d]$.  This area is assumed to be
sufficiently far away from where the main bathymetry change occurs such
that we can consider $h+B_0$ rather than $h+B$.
Hence, for a measurement time $t\in[0,T_F]$, we define
$f^{ob}$ as
\begin{equation}
\label{eq:AP-f^ob}
f^{ob}(h,v;B,t):=\fint_c^d[h(x,t)+B_0(x)]dx:=\frac{1}{|d-c|}\int_c^d [h(x,t)+B_0(x)]dx,
\end{equation}
where $h$ and $v$ are the solutions of shallow water equations
\eqref{eq:AP-PDE} for given $B$, and
$\fint_c^d$ is the average of the integral over $[c,d]$. Since we do
not know exactly at what time $t$ the tsunami wave is close to
shore, we take the maximum over the time interval, resulting in the
parameter-to-event map $F: B \mapsto \bar F(h(B),v(B);B)$, where $\bar
F$ is defined as
\begin{equation}
\label{eq:AP-F}
\bar F(h,v;B) := \max_{t\in[0,T_F]} f^{ob}(h,v;B,t)= \max_{t\in[0,T_F]} \fint_c^d [h(x,t)+B_0(x)]dx.
\end{equation}
In the definition of $F$,  we
consider the variables $h$ and $v$ functions of $B$ through the
solution of the shallow water equations.
Thus, the probability we aim at estimating is the probability that the
maximum average wave height in $[c,d]$ exceed a threshold $z$,
where $B$ follows the distribution introduced in \cref{subsec:tsunamiB}.

The function $\bar F$ \eqref{eq:AP-F} involves the
$\max$-function, which makes optimization difficult. Thus, for
$\gamma>0$  we introduce the regularized parameter-to-event map
$F_\gamma: B\mapsto \bar F_\gamma(h(B),v(B);B)$, where
\begin{equation}
\label{eq:AP-F regularized}
\begin{aligned}
\bar F_\gamma(h,v;B):=
\gamma\log\left[ 
\dfrac{1}{T_F}\int_{0}^{T_F}\exp \left(\dfrac{1}{\gamma}\fint_c^d(h+B_0)dx\right) dt\right].
\end{aligned}
\end{equation}
The smaller $\gamma$, the better \eqref{eq:AP-F regularized} approximates
\eqref{eq:AP-F}. In particular,
\begin{equation}
\lim_{\gamma\rightarrow0}\bar F_\gamma(h,v;B)=	\lim_{\gamma\rightarrow0}\gamma\log\left[ 
\dfrac{1}{T_F}\int_{0}^{T_F}\exp \left(\dfrac{f^{ob}(h,v;B,T)}{\gamma} \right) dt\right]=\max\limits_{t\in[0,T_F]}f^{ob}(h,v;B,T).
\end{equation}

\subsection{LDT-optimization}
Given the parameter space, the governing equations and the event
measure, we now detail the LDT optimization problem
\eqref{eq:LDT-min} over the parameter $B\in\mathcal{B}$.
For the tsunami problem, $I(B)$ and $F(B)$ are defined in \eqref{eq:AP-I} and
\eqref{eq:AP-F} (or \eqref{eq:AP-F regularized}), respectively. The
parameter-to-event map $F$ involves the PDE \eqref{eq:AP-PDE} with
zero initial conditions and proper boundary conditions, which we omit
in the following discussions for brevity.
Since we consider the two parameter-to-event maps \eqref{eq:AP-F} and
\eqref{eq:AP-F regularized}, we obtain two LDT optimization
problems.
\paragraph{Regularized objective} Using the regularization
parameter-to-event map \eqref{eq:AP-F regularized}, the LDT problem is
the PDE-constrained optimization problem
\begin{equation}
  \label{eq:AP-min regularized}
  \begin{aligned}
    \min_{B,h,v}\ & I(B)-\lambda \bar F_{\gamma}(h,v;B), \\
    \text{subject to }\ &  \text{the PDE constraints \eqref{eq:AP-PDE}}.
  \end{aligned}
\end{equation}
For subsequent use, we define the reduced objective $  J_{\gamma, \lambda} (B):=I(B)-\lambda F_\gamma(B)$.
Thus, the PDE-constrained problem \eqref{eq:AP-min regularized} can be
written as unconstrained optimization problem over $B\in \mathcal{B}$.
While the objective $J_{\gamma,\lambda}(\cdot)$ is smooth, its
accurate evaluation can become difficult for small $\gamma>0$, and its
gradients can be large. An alternative to this regularized objective is
to consider the time of the largest average wave height close to shore
as an additional unknown, resulting in the second problem.

\paragraph{Time-optimal problem}
We can also consider a time-optimal LDT optimization that does not require
a regularization parameter $\gamma$. Using the
definition of $\bar F$ in \eqref{eq:AP-F}, additional optimization over the
time results in the PDE-constrained optimization problem
\begin{equation}
  \label{eq:AP-min time-opt}
  \begin{aligned}
    \min_{\begin{subarray}{c} B, h,v,\\ t\in[0,T_F] 
    \end{subarray}}&\ I(B)-\lambda  f^{ob}(h,v;B,t), \\ 
    \text{subject to }&\ \text{the PDE constraints \eqref{eq:AP-PDE}}.
  \end{aligned} 
\end{equation}
The corresponding reduced objective is $J_{\lambda}(B,t):=I(B)-\lambda  f^{ob}(h(B),v(B);B,t)$,
where $h(B)$ and $v(B)$ are again the solutions of shallow water equations \eqref{eq:AP-PDE}.

\subsection{Discretization and stabilization}\label{subsec:discretization}

To solve the optimization problems \eqref{eq:AP-min regularized} and
\eqref{eq:AP-min time-opt} numerically, we have to discretize the
continuous functions $B$, $v$, $h$ together with the governing
equations.
Since the shallow water equations \eqref{eq:AP-PDE vis} are
hyperbolic,
we use a discontinuous Galerkin finite element method (DG-FEM)
\cite{hesthaven2007nodal} to discretize the equations in space. For
discretization in time, we use a Runge-Kutta scheme.

Since the shallow water equations \eqref{eq:AP-PDE} are a system of
nonlinear hyperbolic equations, the solution can have shocks, i.e.,
the slope of the solution variables can become infinite. It is well
known that the numerical approximation of systems with shocks is
challenging \cite{LeVeque02}. This is even more compounded for
adjoint-based derivative computation. Some of the discretization and
stabilization choices we make here are in fact motivated by our focus
on adjoint-based derivatives, as will become clear in the subsequent
subsections.
Partially motivated by the need for well-defined discrete adjoint
equations (see \cref{subsec:adj-discrete}), we
add artificial viscosity to the shallow water equations
\eqref{eq:AP-PDE} to prevent slopes that cannot be resolved by the
discretization. There are different approaches of adding artificial
viscosity to the shallow water equations. One is adding viscosity for
both the mass and momentum conservation laws
\cite{chen2013adaptive}. Here, we only add viscosity to the momentum
equation, as discussed in \cite{mascia2006asymptotic}, where the
authors prove that the solutions of the resulting system preserves
stationary steady states and is asymptotically stable. This modified
problem is
\begin{subequations}
\label{eq:AP-PDE vis}
\begin{alignat}{3}
  h_t+v_x&=0 \quad && \text{ on } \mathcal{D}\times [0,T_F],\label{eq:AP-PDE vis1}\\
 v_t+\left(\frac{v^2}{h}+\frac{1}{2}gh^2-\epsilon h
 \varphi\right)_x+ghB_x&=0 && \text{ on } \mathcal{D}\times [0,T_F], \label{eq:AP-PDE vis2}\\
 \varphi+\left(-\frac{v}{h}\right)_x&=0 && \text{ on } \mathcal{D}\times [0,T_F], \label{eq:AP-PDE vis3} %
\end{alignat}
\end{subequations}
with the initial and boundary conditions \eqref{eq:AP-PDE3} and
\eqref{eq:AP-PDE4}.  Here, $\varphi(x,t)$ serves as an auxiliary variable
which allows to write the dissipative operator in a way suitable for a
DG scheme. The parameter $\epsilon$ controls how much artificial
viscosity is added, and we choose $\epsilon=O(|\bar{h}|)$ with
$\bar{h}$ being the element length as proposed in
\cite{ulbrich2019numerical,LeVeque02}.

Our implementation uses a DG discretization with linear interpolating
polynomials in space.
For \eqref{eq:AP-PDE vis1} and \eqref{eq:AP-PDE vis2}, we use a global
Lax-Friedrichs flux of the form
\begin{equation}
\label{eq:AP-LF flux}
f^*(q)=
\dfrac{f(q^-)+f(q^+)}{2}+\frac{C^{LF}}{2}n^-(q^--q^+), 
\end{equation}
where $q$ stands for either $h$ or $v$. Moreover, $f(q)$ is the
corresponding flux, $+$ and $-$ denote the exterior and
the interior value at each element interface,
and $C^{LF}$ is the global Lax-Friedrichs constant. A
less diffusive alternative to a global Lax-Friedrichs flux would be
a local variant, where the flux at each interface depends on the
state variable. While using such a local flux in the context of
adjoint equations might be possible, here we prefer to avoid
technical challenges and possible inconsistencies and use
the same global Lax-Friedrichs constant $C^{LF}$ for all elements:
\begin{equation}
C^{LF}=\max \left(  \left|\dfrac{v}{h}\right|+\sqrt{gh}\right).
\end{equation}
For \eqref{eq:AP-PDE vis3},
we use a central flux in
the DG scheme, i.e., the average of the values at the interfaces.
Although the numerical results presented in this paper use a
first-order DG scheme, the proposed method can be generalized to
higher-order spatial discretizations. To discretize in time, we use a
strong stability-preserving second-order Runge-Kutta (SSP-RK2) method
\cite{hesthaven2007nodal}. The strong stability-preserving (SSP) property guarantees preservation of
the total variation of the discrete solution.

\subsection{Adjoint-based gradient computation}\label{sec:adjoint}
Since the objectives $J_{\lambda}(\cdot)$ and
$J_{\gamma,\lambda}(\cdot)$ 
require the solution of a PDE, we use
adjoints to efficiently compute their derivatives
\cite{borzi2011computational,troltzsch2010optimal,Reyes15,HinzePinnauUlbrichEtAl09}. Here, we present
the continuous form of these adjoint equations. Their discretization
is summarized in \cref{subsec:adj-discrete}. We skip details of the
technical derivation and only present the results, starting with the
regularized objective.
\paragraph{Regularized objective}
To derive the adjoint system for the shallow water equations with
artificial viscosity \eqref{eq:AP-PDE vis}, we use a formal Lagrangian
approach, i.e., we define the Lagrangian as the sum of the objective
and the weak form of the state equations, where the test functions
take the role of the Lagrange multiplier functions. Then, setting
variations with respect to the state variables in all directions to
zero results in the adjoint equations in the unknowns $(p,w, \psi)$:
\begin{subequations}
\label{eq:AP-OTD adj reg}
\begin{alignat}{3}
p_t+\left(-\frac{v^2}{h^2}+gh-\epsilon
\varphi\right)w_x-\frac{v}{h^2}\psi_x -gB_xw+\lambda \partial_h
\bar F_{\gamma}&=0 \quad && \text{ on } \mathcal{D}\times [0,T_F],\label{eq:AP-OTD adj reg1}\\
 w_t +p_x+\frac{2v}{h}w_x-\frac{1}{h}\psi_x&=0 && \text{ on } \mathcal{D}\times [0,T_F],\label{eq:AP-OTD adj reg2}\\
 \psi-\epsilon h w_x&=0 && \text{ on } \mathcal{D}\times [0,T_F],\label{eq:AP-OTD adj reg3}\\
 p(x, T_F)=0, \ w(x,T_F)&=0 && \text{ for } x\in\mathcal{D},\label{eq:AP-OTD adj reg4}\\
 w(a,t)=w(b,t)&=0 && \text{ for } t\in[0,T_F].\label{eq:AP-OTD adj reg5}
\end{alignat}
\end{subequations}
Here,
the partial derivative of $\bar F_\gamma$ with
respect to $h$ is defined as
\begin{equation*}
  \partial_h \bar F_{\gamma}:=\frac{1}{T_F}\exp\left\lbrace \frac{1}{\gamma}\left[ \fint_c^d(h+B_0)dx-\bar F_\gamma\right]  \right\rbrace,\, \text{on } [c,d]\times [0,T_F],
\end{equation*}
and $\partial_h \bar F_{\gamma}:=0$ else.  When
solving the adjoint system \eqref{eq:AP-OTD adj reg}, the
state variables $(v,h,\varphi)$ are know and we only solve for the
adjoint variables $(p,w,\psi)$, which appear linear in
\eqref{eq:AP-OTD adj reg}. Note that due to \eqref{eq:AP-OTD adj
  reg4}, this is a final value problem that must be solved backwards in
time. Once the state and the adjoint variables are know, one can
obtain the derivative $\mathcal G(B)(\hat B)$ of $J_{\gamma,\lambda}$
in an arbitrary direction $\hat{B} = O\hat S$ as the variation of the Lagrangian
with respect to $B=B_0+OS$ in that direction, i.e.
\begin{equation}
\label{eq:AP-OTD gradient}
\mathcal{G}(B)(\hat{B})=
\langle\!\!\langle  B-B_0, \hat{B}\rangle\!\!\rangle_{C_s}
+\int_{0}^{T_F}\int_{\mathcal{D}}ghw \hat{B}_x dxdt = S^\top C_s^{-1} \hat S
+\int_{0}^{T_F}\int_{\mathcal{D}}ghw (O\hat S)_x dxdt.
\end{equation}

\paragraph{Time-optimal objective}
For the time-optimal problem \eqref{eq:AP-min time-opt}, additionally to
the derivative with respect to $B$, we require derivatives with
respect to the observation time $t$.
Again, we skip details here---optimization over time or
time-optimal control is a challenging research topic by
itself \cite{Fattorini05,KunischRund15}.

The main difference between $J_{\lambda}$ and $J_{\gamma,\lambda}$ is
that in the latter, $\bar F_\gamma(h,v;B)$ is replaced by
$f^{ob}(h,v;B,t)$. Thus, one obtains the adjoint equations for the
time optimal problem \eqref{eq:AP-min time-opt} by replacing 
$\partial_h \bar F_{\gamma}$ in \eqref{eq:AP-OTD adj reg} with the
derivative of $f^{ob}$ with respect to $h$, i.e.,
$\partial_h f^{ob}:=
{1/|d-c|}$ on $[c,d]\times[0,T_F]$ and $\partial_h f^{ob}:=0$ else.
Additionally, the final time conditions becomes
$p(x,T_F)=
\lambda/{|d-c|}$ for $x\in[c,d]$ and $p(x,T_F)=0$ else.
Since $F_\gamma$ and $f^{ob}$ do not depend explicitly on $B$, the
gradient of $J_{\lambda}$ %
is identical to \eqref{eq:AP-OTD gradient}.

Finally, we require the derivative of $J_\lambda$
with respect to the observation time $t$. A short
computation yields that
\begin{equation}
\label{eq:AP-J_T}
\frac{\partial}{\partial t}J_{\lambda}(B,t)
=-\lambda \frac{\partial}{\partial t}f^{ob}(h,v;B,t)%
=-\lambda \fint_c^d \frac{\partial}{\partial t}h(x,t)dx
=\lambda \fint_c^d \frac{\partial}{\partial x}v(x,t) dx,
\end{equation}
where the last identity follows from the conservation-of-mass equation
$h_t+v_x=0$.

\subsection{Discretization of adjoint equations and gradient}\label{subsec:adj-discrete}
When shocks occur in the state equations, this may lead to
discontinuous coefficients in the adjoint equations.  Thus, the theory
and grid convergence of adjoint-based gradients for hyperbolic systems
is challenging and rigorous results are rare. The authors of
\cite{giles2010convergence} study the grid convergence of the adjoint
solutions for Burger's equation, and find that solutions of the finite
difference-discretized equation may converge to a wrong continuous
solution when the state solution has shocks. To smooth out shocks that
cannot be resolved by the mesh, they propose adding artificial
viscosity that vanishes at a certain rate as the mesh is refined. The
result on the required rate has been improved recently
\cite{ulbrich2019numerical}.
As discussed in \cref{subsec:discretization}, we follow a similar
strategy in the context of a discontinuous Galerkin discretization for
the shallow water equations.

To discretize the adjoint equations and the gradient expressions
from the previous section, we follow a
discretize-then-optimize approach, i.e., we first discretize the
optimization objective and the governing equations in space and time,
and then compute discrete derivatives. This means that the
discretization of the adjoint equation is implied by that of the state
equation. An alternative would be the optimize-then-discretize
approach, which discretizes the continuous adjoint equation
independently. While more convenient, this may result in inconsistent
gradients, i.e., numerically gradients that are not exact gradients of
any discrete (or continuous) problem. Both approaches have their
advantages and disadvantages, but here we follow the former approach,
i.e., discretize the problem and then compute the corresponding
adjoint-based gradient. In the previous section we nevertheless
presented the continuous adjoint equations to show and discuss their
structure.  We suppress the (interesting)
technical details of the following computations for space reasons, and only summarize the
results.

Following this discretize-then-optimize approach, we find that the
adjoint of the spatial DG-discretization of \eqref{eq:AP-PDE vis} is
again a DG discretization of the continuous adjoint equations,
extending results in \cite{wilcox2015discretely} to nonlinear
conservation laws. The induced flux in the adjoint equations is a
modified global Lax-Friedrichs flux.
We follow the same discretize-then-optimize approach for the
Runge-Kutta time discretization.  Results in \cite{hajian2019total}
show that the SSP property for the state equation ensures stability of
the discrete adjoint time-stepping scheme. While the adjoint
time-adjoint method does not coincide with the SSP-RK2 scheme, it is
also a second-order scheme that preserves stability.  Since the
regularized objective $\bar F_\gamma$ involves integration over time and we
use the quadrature induced by SSP-RK2 for its discretization.  The
bathymetry $B$ is discretized using linear continuous finite
elements. The embedding of linear continuous to discontinuous elements
as needed in \eqref{eq:AP-PDE vis} is trivial, and the adjoint of this
embedding is used to transfer the gradient from the discontinuous to
the continuous space.

Due to the use of a DG scheme and the
discretize-then-optimize approach, the gradient expressions include
additional terms at element interfaces, as observed for linear
problems \cite{wilcox2015discretely}. These additional terms
vanish in the limit as the mesh is refined, but they must be included to
obtain exact gradients of the discretized problem. To avoid the
technical derivations, we only present the continuous forms of the
gradient in \eqref{eq:AP-OTD
  gradient}. We verify the correctness of our gradient implementation,
by comparing directional derivatives with their finite differences
approximations. Due to the discretize-then-optimize approach, they
coincide not only for physics-resolving, but also for coarse
meshes up to what can be expected in the presence of machine round-off.

\section{Results for tsunami problem}\label{sec:results}
Here, we study the convergence behavior of the proposed
algorithms and approximations. We also discuss qualitative results
such as the bathymetry change resulting in the most extreme tsunami
event and extreme event probabilities. First, we discuss the numerical
solution of the LDT optimization problems.

\subsection{Shallow water equation-constrained optimization}
To compute minimizers for \eqref{eq:LDT-min}, we need to solve the
PDE-constrained optimization problems \eqref{eq:AP-min regularized}
and \eqref{eq:AP-min time-opt}. We use the adjoint method discussed
in \cref{sec:adjoint} to compute gradients and use a preconditioned
steepest descent method for the optimization.  Backtracking line search
using the Armijo rule \cite{nocedal2006numerical} is used for
globalization of the descent algorithm.  We precondition the gradient
with the covariance matrix.

In \cref{tab:ite}, we present iterations numbers for different values
of $\lambda$, as well as the corresponding extreme event values and
probability estimates based on the second-order approximation
discussed in \cref{sec:SORM}.  For each $\lambda$, we take the
reference bathymetry $B_0$ as the starting point for the
optimization.
We observe in \cref{tab:ite} that
the iteration numbers are generally insensitive to $\lambda$ for both the regularized and the time-optimal problem. Since larger $\lambda$'s
correspond to extremer events, we find in particular that the
number of iterations is independent of the extremeness of events.
This is a desirable property that often does not hold for
sampling-based methods.

\begin{table}[tb] %
	{\footnotesize
		\caption{Number of iterations for different
                  $\lambda$'s, for optimization with regularized
                  objective $F_\gamma$ in \eqref{eq:AP-min
                    regularized} with $\gamma=0.003$, and with
                  time-optimal objective \eqref{eq:AP-min
                    time-opt}. The iteration is terminated when the
                  $\invC$-weighted norm of the gradient is reduced by
                  5 orders of magnitude. Shown are also the values of
                  $z=z(\lambda)$ and the event probability estimate
                  computed using a second-order approximation of
                  $\varOmega(z)$.
                }\label{tab:ite}
		\begin{center}
			\begin{tabular}{|c||c|c|c||c|c|c|} \hline
			\multirow{2}{*}{$\lambda$}	 &\multicolumn{3}{c||}{Regularized objective $F_\gamma$} & \multicolumn{3}{c|}{Time-optimal problem} \\ 
		\cline{2-7}	\\[-1.2em] 
			& $z:=F_\gamma(B^\star(\lambda))$ & $P^{SO}(z)$&\# iter& $z:=F(B^\star(\lambda))$ & $P^{SO}(z)$&\# iter\\ 	 \hline
				12 & 0.263  &4.80e-02 & 23& 0.281  & 4.70e-02 & 35\\
				16 & 0.364 & 9.55e-03 &31& 0.382  &9.36e-03   &27\\
				20& 0.468 &1.24e-03 &24& 0.486 &1.22e-03  &20\\
				24& 0.574 &1.04e-04 &31& 0.592 &1.02e-04  &20\\
				28& 0.682&5.45e-06 &27& 0.701  &5.33e-06  &30\\
				32& 0.792&1.77e-07 & 33& 0.811 &1.73e-07 & 27\\
				36&0.905& 3.54e-09 &29&0.923 & 3.45e-09  &34\\
				40& 1.018&4.27e-11 &32& 1.037 &4.17e-11 &38\\
				44& 1.134 &3.09e-13 &30& 1.152 &3.02e-13  &30\\
				48&1.250 &1.36e-15 &37 &1.269 &1.26e-15 &35 \\ \hline
			\end{tabular}
		\end{center}
	}
\end{table}

\begin{figure}[tbp]\centering
	\begin{tikzpicture}[]
	 \begin{groupplot}[group style = {columns=2, horizontal sep = 10pt}, width = 5.5cm, height = 5.0cm]
	 \nextgroupplot[compat=1.3, scale only axis,
	 xlabel={Distance to shore [km]},
	 ylabel={Optimal $B-B_0$ [m]},
	  legend style = {font=\tiny,nodes=right, legend to name=grouplegend},legend columns=4, 
	 xmin=165,xmax=195]
	\addplot [color=\cOne, line width=1]
	table[x expr=\thisrowno{0}*0.001,y=B] {\data/Toho_TmOpt_optsol12std10nos20.txt};
	\addlegendentry{Time-opt, $\lambda=12$\phantom{00}}
	\addplot [color=\cTwo, mark=none, mark size=2pt, line width=1]
	table[x expr=\thisrowno{0}*0.001,y=B] {\data/Toho_TmOpt_optsol24std10nos20.txt};
	\addlegendentry{Time-opt, $\lambda=24$\phantom{00}}
	\addplot [color=\cThree, mark=none, mark size=2pt, line width=1]
	table[x expr=\thisrowno{0}*0.001,y=B] {\data/Toho_TmOpt_optsol36std10nos20.txt};
	\addlegendentry{Time-opt, $\lambda=36$\phantom{00}}
	\addplot [color=\cFour, mark=none, mark size=2pt, line width=1]
	table[x expr=\thisrowno{0}*0.001,y=B] {\data/Toho_TmOpt_optsol48std10nos20.txt};
	\addlegendentry{Time-opt, $\lambda=48$\phantom{00}}
	\addplot [color=black!100, mark=none, mark size=2pt, line width=.5,
	densely dotted]
	table[x expr=\thisrowno{0}*0.001,y=B] {\data/Toho_Approx_optsol12std10nos20.txt};
	\addlegendentry{$\gamma$-reg, $\lambda=12$}
	\addplot [color=black!100, mark=none, mark size=2pt, line
	width=0.5, densely dashed]
	table[x expr=\thisrowno{0}*0.001,y=B] {\data/Toho_Approx_optsol24std10nos20.txt};
	\addlegendentry{$\gamma$-reg, $\lambda=24$}
	\addplot [color=black!100, mark=none, mark size=2pt, line
	width=.5, dotted]
	table[x expr=\thisrowno{0}*0.001,y=B] {\data/Toho_Approx_optsol36std10nos20.txt};
	\addlegendentry{$\gamma$-reg, $\lambda=36$}
	\addplot [color=black!100, mark=none, mark size=2pt,
          line width=.5,densely dashdotted]
	table[x expr=\thisrowno{0}*0.001,y=B] {\data/Toho_Approx_optsol48std10nos20.txt};
	\addlegendentry{$\gamma$-reg, $\lambda=48$}
\nextgroupplot[scale only axis,
	xlabel={Distance to shore [km]},
	ylabel={Optimal slip [m]}, ,ylabel near ticks, yticklabel pos=right,
	 xmin=178, xmax=187
]
	\addplot+[jump mark left,mark =none] [color=\cOne, line width=1]
	table[x expr=\thisrowno{0}*0.001,y=slip] {\data/Toho_TmOpt_slip12std10nos20.txt};
	\addplot+[jump mark left,mark =none] [color=\cTwo, line width=1]
	table[x expr=\thisrowno{0}*0.001,y=slip] {\data/Toho_TmOpt_slip24std10nos20.txt};
	\addplot+[jump mark left,mark =none] [color=\cThree, line width=1]
	table[x expr=\thisrowno{0}*0.001,y=slip] {\data/Toho_TmOpt_slip36std10nos20.txt};
	\addplot+[jump mark left,mark =none] [color=\cFour, line width=1]
	table[x expr=\thisrowno{0}*0.001,y=slip] {\data/Toho_TmOpt_slip48std10nos20.txt};
	\addplot+[jump mark left,mark =none] [color=black!100, mark=none, mark size=2pt, line width=.5,
	densely dotted]
	table[x expr=\thisrowno{0}*0.001,y=slip] {\data/Toho_Approx_slip12std10nos20.txt};
	\addplot+[jump mark left,mark =none] [color=black!100, mark=none, mark size=2pt, line width=.5,
	densely dashed]
	table[x expr=\thisrowno{0}*0.001,y=slip] {\data/Toho_Approx_slip24std10nos20.txt};
	\addplot+[jump mark left,mark =none] [color=black!100, mark=none, mark size=2pt, line width=.5,
	dotted]
	table[x expr=\thisrowno{0}*0.001,y=slip] {\data/Toho_Approx_slip36std10nos20.txt};
		\addplot+[jump mark left,mark =none] [color=black!100, mark=none, mark size=2pt, line width=.5,
		densely dashdotted]
	table[x expr=\thisrowno{0}*0.001,y=slip] {\data/Toho_Approx_slip48std10nos20.txt};
	\end{groupplot}
\node[black] at ($(group c2r1) + (-3.0cm,3.0cm)$) {\pgfplotslegendfromname{grouplegend}}; 
	\end{tikzpicture}
	\caption{Shown on the left are optimal bathymetry changes of
          LDT-solutions $B^\star$ for different $\lambda$'s (time
          optimal and regularized objective $F_\gamma$ with
          $\gamma=0.003$). For fixed $\lambda$, the optimizers of the
          two problems are quite similar, showing that the
          approximation of the $\max$-function with $F_\gamma$ is
          quite effective. Shown on the right are fault slips
          corresponding to the optimal solutions $B^\star$ 
          for different $\lambda$'s as discussed in \cref{subsec:tsunamiB}.}\label{fig:opt}
\end{figure}
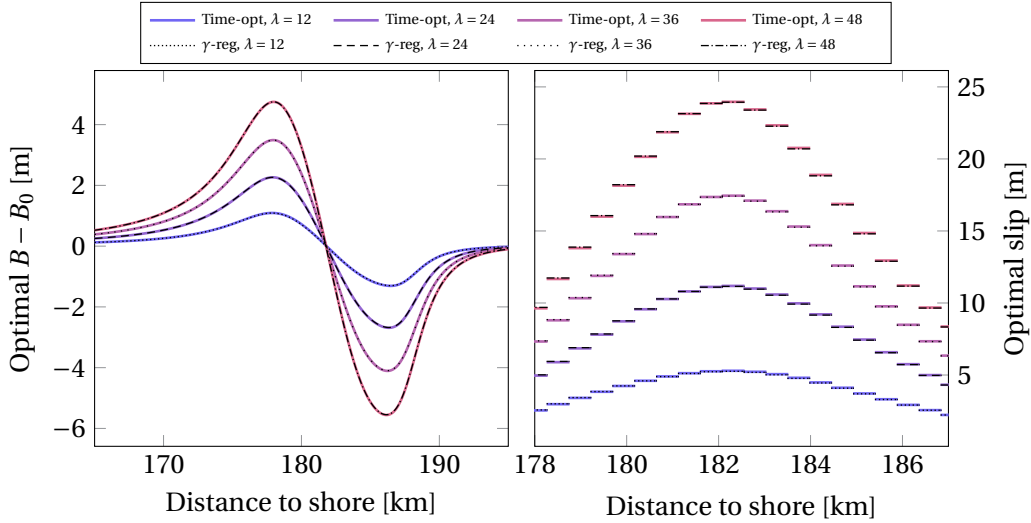

\Cref{fig:opt} shows the optimal bathymetry changes $B^\star-B_0$ for
different values of $\lambda$, and thus different extreme event
thresholds $z$. We show results for the regularized and the
time-optimal formulations \eqref{eq:AP-min regularized} and
\eqref{eq:AP-min time-opt}. Since $\gamma$ is chosen rather small,
there is visually 
little difference between the optimizers found with these different
formulations.  As can be seen, the most effective mechanism for large
tsunamis on shore involves an uplift in the shore-facing part and a
downlift away from the shore. The corresponding slips generating these
bathymetry changes can be seen on the right in \Cref{fig:opt}. The 20 slip
patches all move in the same direction and the slip is larger in the
middle than at the sides of the slip area. Since tsunami
waves interact with the bathymetry, these optimal patters depend, at
least to some degree, on the structure of the bathymetry and the
location where the event is observed.

Note that optimizers for different $\lambda$ have a similar structure
but their magnitude varies with the extremeness of the
event.  To explain these magnitude differences, recall that the rate
function $I$ is quadratic. If the parameter-to-event map $F$ were
linear, then the LDT minimizer would increase linearly with $\lambda$
as can be seen from the optimality conditions of such a quadratic
optimization objective. Deviations from that scaling are a result of
the nonlinearity in the parameter-to-event map caused by the
nonlinearity of the shallow water equation and the extreme event
objective. Since this deviation is small, we deduce that the
problem is moderate nonlinear. This (together with the results
presented in the subsequent \cref{sec:UQresults}) indicates a
posteriori that the assumptions needed for our LDT theory are likely
satisfied in this problem.

\subsection{Eigenvalue estimation for second-order approximation $P^{SO}(z)$} \label{sec:low-rank} As discussed in
\cref{sec:SORM}, computing the prefactor using \eqref{eq:PE-SORM
  Gauss} requires estimation of the eigenvalues of the Hessian of the
parameter-to-observable map, preconditioned with the covariance of the
Gaussian parameter distribution, i.e.,
$(OC^{1/2}_s)^\top\nabla^2_BF(B^\star)O C^{1/2}_s$.  Here, we study
the feasibility of this approach for the tsunami problem. In these
numerical tests we approximate the Hessian-application
using finite differences of gradients.

As discussed in \cref{subsec:tsunamiB}, the random parameter $B$ is
modeled using 20 slips at the fault boundary below the ocean floor.
Thus, and due to typical properties of covariance matrices, we argued in
\cref{subsec:low-rankSORM}  that the eigenvalues of this
preconditioned Hessian decay rapidly.
To verify this numerically, we compute the eigenvalues of
preconditioned Hessians for different $\lambda$'s and multiply them by
$\lambda$ as in \cref{thm:SO}.
The results for the tsunami problem are shown in \Cref{fig:lambda*eig std10nos20}.
It can be seen that the eigenvalues decay rapidly and this behavior
barely changes with the extremeness of the event. This shows that it
is sufficient to use a small number of dominating eigenvalues in the
second-order approximation. However, the largest value of about 0.5
indicates non-negligible nonlinearity of the parameter-to-event map
$F$. If $F$ were linear, all eigenvalues would be zero. In addition,
we find that all leading eigenvalues are positive, indicating that $F$
is convex in all leading directions close to the LDT-minimizers.
This results in a larger-than-one multiplicative SORM-correction term
\cref{eq:PE=SORM k_i}. Thus, the probability estimate from the
first-order approximation is smaller than the estimate from  the
second-order approximation.

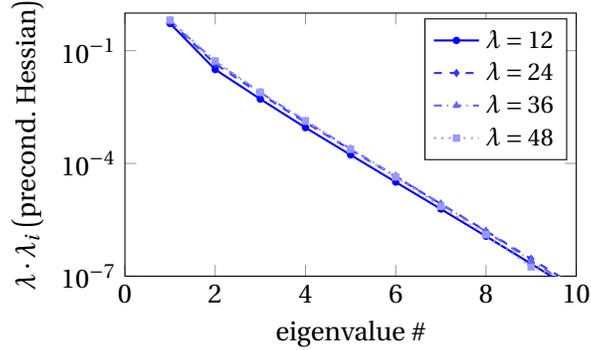
\begin{figure}[tb]\centering
	\begin{tikzpicture}[]
	\begin{semilogyaxis}[compat=newest, width=6cm, height=3.5cm, scale only axis,
	xlabel=eigenvalue \#,
	ylabel=$\lambda\cdot\lambda_i\left( \text{precond.\ Hessian}\right) $,
	legend style={font=\small,nodes=right}, xmin=0, xmax=10,
        ymax=1, ymin=1e-7]
	  \addplot [color=blue!100!white, mark=*, mark size=1pt,
            thick, mark options={solid}]
	table[x=x,y=y] {\data/lambda_eig_lambda12.txt};
	\addlegendentry{$\lambda=12$}
	\addplot [color=blue!80!white, mark=diamond, mark size=1pt,
          dashed, thick, mark options={solid}]
	table[x=x,y=y] {\data/lambda_eig_lambda24.txt};
	\addlegendentry{$\lambda=24$}
	\addplot [color=blue!60!white, mark=triangle*, mark size=1pt,
          dashdotted, thick, mark options={solid}]
	table[x=x,y=y] {\data/lambda_eig_lambda36.txt};
	\addlegendentry{$\lambda=36$}
	\addplot [color=blue!40!white, mark=cube*, mark size=1pt,
          dotted, thick, mark options={solid}]
	table[x=x,y=y] {\data/lambda_eig_lambda48.txt};
	\addlegendentry{$\lambda=48$}
	\end{semilogyaxis}
	\end{tikzpicture}
	\caption{Shown are the dominating
          eigenvalues of the preconditioned Hessian multiplied with
          the corresponding $\lambda$ defined in \cref{thm:SO} for
          various values of $\lambda$. The eigenvalues that are small
          compared to 1 have little influence on $P^{SO}(z)$, i.e.,
          computation of about 5 eigenvalues is sufficient in our
          example.
          Note that the rapid decay is insensitive to $\lambda$, and thus
          to how extreme the event is.} \label{fig:lambda*eig std10nos20}
\end{figure}

\begin{figure}[tbhp]\centering
  \begin{tikzpicture}[spy using outlines=
      {rectangle,lens={scale=2}, size=8cm, connect spies}]
	\begin{semilogyaxis}[compat=newest, width=8cm, height=5.5cm, scale only axis,
	xlabel={Threshold $z$ [m]},
	ylabel=Probability $\mathbb{P}(F_\gamma\geq z)$,
	xmin=0, xmax=1.2, ymin=1e-14,ymax=1,
	legend style={font=\tiny,nodes=right}]
          \addplot [color=\cMC, line width=1]
	table[x=zz,y=mean_prob] {\data/Toho_Approx_meanstd10nos20.txt};
	\addlegendentry{MC sampling}
	\addplot [name path= MCCI1,color=\cMCci, line
          width=0.8,dashed,forget plot]
	table[x=zz,y=mean_prob_CI] {\data/Toho_Approx_mean_CIstd10nos20.txt};
	\addplot [name path= MCCI2,color=\cMCci, line width=0.8,dashed]
	table[x=zz,y=mean_prob_CI] {\data/Toho_Approx_mean_CI2std10nos20.txt};
	\addlegendentry{MC 95\% CI}
	\addplot [color=\cFit, mark=triangle*, mark size=1.5pt,
          line width=0.5, mark options={solid}]
	table[x=F,y=prob] {\data/Toho_Approx_fitstd10nos20.txt};
	\addlegendentry{Fitting of $\exp(-I)$}
	\addplot [color=\cFORM, mark=square*, mark size=1.5pt, line
          width=0.5, mark options={solid}]
	table[x=F,y=prob] {\data/Toho_Approx_FORMstd10nos20.txt};
	\addlegendentry{First-order approx.}
	\addplot [color=\cSORM, mark=*, mark size=1.5pt,
          line width=0.5, mark options={solid}]
	table[x=F,y=prob] {\data/Toho_Approx_SORMstd10nos20.txt};
	\addlegendentry{Second-order approx.}
		\addplot [color=cyan, mark=none, mark size=1.5pt,
	line width=0.8, mark options={solid}, densely dotted]
	table[x=F,y=prob] {\data/Toho_Approx_Linearstd10nos20.txt};
	\addlegendentry{Linearized $F_\gamma$}
        \coordinate (spypoint) at (0.61,3e-5);
        \coordinate (magnifyglass) at (0.24,1e-10);
	\end{semilogyaxis}
        \spy [gray, size=2.5cm] on (spypoint)
        in node[fill=white] at (magnifyglass);
  \end{tikzpicture}
  \caption{Comparison of probability estimation for regularized
    objective $F_\gamma$ \eqref{eq:AP-F regularized} with
    $\gamma=0.003$. Shown in blue are the mean and 95\% confidence
    intervals obtained with standard MC with \numosamples samples
    (discussed in \cref{sec:MC}), in purple results obtained by
    fitting the asymptotic LDT rate with the MC mean
    (\cref{sec:fitting}), and results using first-order and
    second-order approximation of $\varOmega(z)$
    (\cref{sec:FORM,sec:SORM}) in red and yellow, respectively. Each
    marker represents the solution of an LDT optimization problem with
    a different value of $\lambda$. The zoom-in shows the regime where
    the variance of the standard MC sampling method increases and
    standard MC sampling becomes infeasible. For comparison, the cyan
    dotted line shows the probabilities
    obtained by linearization of $F_\gamma$ at the optimizer
    $B^\star$ for $\lambda=12$.
  }\label{fig:prob reg}
\end{figure}
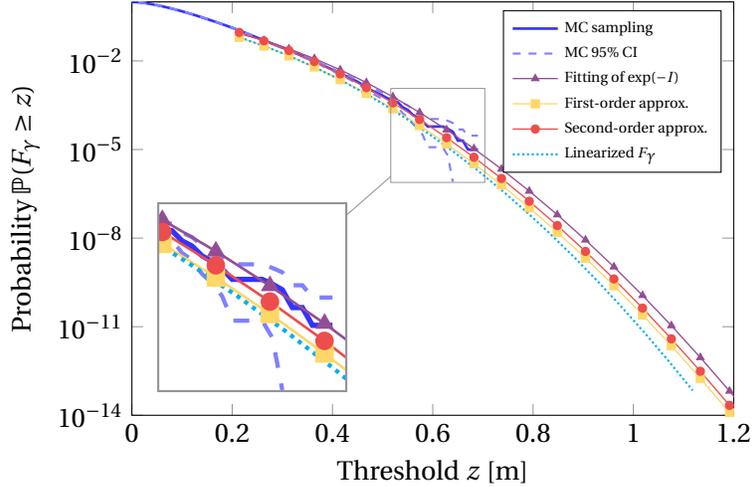

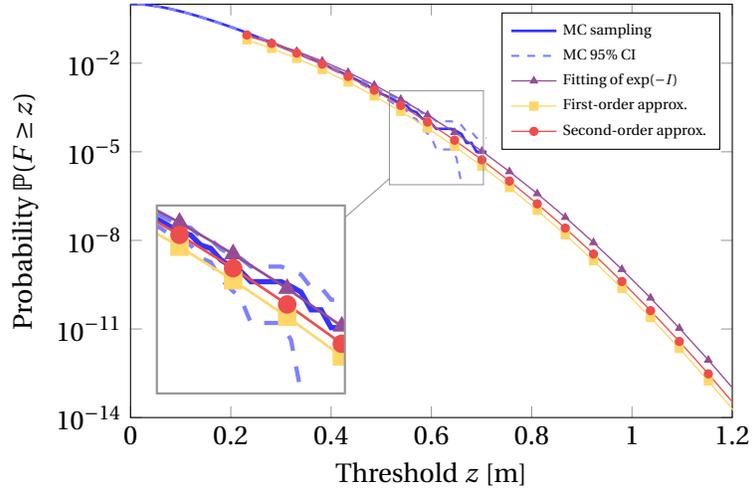
\begin{figure}[tbhp]\centering
	\begin{tikzpicture}[spy using outlines={rectangle,lens={scale=2}, size=8cm, connect spies}]
	\begin{semilogyaxis}[compat=newest, width=8cm, height=5.5cm, scale only axis,
	xlabel={Threshold $z$ [m]},
	ylabel=Probability $\mathbb{P}(F\geq z)$,
xmin=0, xmax=1.2, ymin=1e-14,ymax=1,
	legend style={font=\tiny,nodes=right}]
	\addplot [color=\cMC, line width=1]
	table[x=zz,y=mean_prob] {\data/Toho_TmOpt_meanstd10nos20.txt};
	\addlegendentry{MC sampling}
	\addplot [color=\cMCci, line width=0.8,dashed,forget plot]
	table[x=zz,y=mean_prob_CI] {\data/Toho_TmOpt_mean_CIstd10nos20.txt};
	\addplot [color=\cMCci, line width=0.8,dashed]
	table[x=zz,y=mean_prob_CI] {\data/Toho_TmOpt_mean_CI2std10nos20.txt};
	\addlegendentry{MC 95\% CI}
		\addplot [color=\cFit, mark=triangle*, mark size=1.5pt,
	line width=0.5, mark options={solid}]
	table[x=F,y=prob] {\data/Toho_TmOpt_fitstd10nos20.txt};
	\addlegendentry{Fitting of $\exp(-I)$}
\addplot [color=\cFORM, mark=square*, mark size=1.5pt, line
width=0.5, mark options={solid}]
	table[x=F,y=prob] {\data/Toho_TmOpt_FORMstd10nos20.txt};
	\addlegendentry{First-order approx.}
	\addplot [color=\cSORM, mark=*, mark size=1.5pt,
line width=0.5, mark options={solid}]
	table[x=F,y=prob] {\data/Toho_TmOpt_SORMstd10nos20.txt};
	\addlegendentry{Second-order approx.}
	        \coordinate (spypoint) at (0.61,3e-5);
	\coordinate (magnifyglass) at (0.24,1e-10);
	\end{semilogyaxis}
	\spy [gray, size=2.5cm] on (spypoint)
	in node[fill=white] at (magnifyglass);
	\end{tikzpicture}
	\caption{Same as \cref{fig:prob reg}, but for time-optimal objective
          $F$ defined in \eqref{eq:AP-F}.
        }\label{fig:prob tim}
\end{figure}

\subsection{Comparison of extreme event quantification methods}\label{sec:UQresults}

In this section, we compare the proposed extreme event estimation
methods for the Tohoku-Oki tsunami.
In the \Cref{fig:prob reg,fig:prob tim}, we compare the results of
Monte Carlo sampling with the LDT approaches (constant prefactor
estimated by fitting with MC data, the first and
second-order approximation of the set $\varOmega(z)$) for both the
regularized objective problem \eqref{eq:AP-min regularized} and the
time optimal problem \eqref{eq:AP-min time-opt}. The reference
probability for moderately extreme events is computed with Monte Carlo
sampling with \numosamples samples using the estimator $P^{MC}_N(z)$
in \eqref{eq:PE-MC estimate P^MC}. This procedure is clearly very
costly in particular when one is interested in extreme events. We also
show the $95\%$ confidence interval for the estimator, which is tight
for $z < 0.4$. However, the Monte Carlo estimator $P^{MC}_N(z)$ only
provides acceptable accuracy for a probability down to about
$10^{-4}$. We also use the LDT logarithmic rate with a constant prefactor as
discussed in \cref{sec:fitting},
fitting the Monte Carlo results in the interval $z\in[0.2,0.4]$.
The resulting estimate seems  to overestimate the
extreme event probability. It also requires MC sampling for estimating
the fitting constant. The first and second-order approximation of
$\varOmega(z)$ do not require fitting since they rely only on the
LDT-optimizers and the local derivative information around the
optimizers. The first-order approximation results in \Cref{fig:prob
  reg,fig:prob tim} are below the Monte Carlo estimator, showing that
significant parts of $\varOmega(z)$ are not contained in the half-space
$\mathcal{H}(z)$. The second-order approximation results in
\Cref{fig:prob reg,fig:prob tim} are closer to the MC estimator,
indicating that the second-order approximation of $\varOmega(z)$
describes the set $\varOmega(z)$ well. All approaches provide
probability estimates down to $10^{-14}$. Comparing the results in
\Cref{fig:prob reg,fig:prob tim} shows that there is little
difference between the time-optimal formulation and the regularization
formulation with $\gamma=0.003$.
In \Cref{fig:prob reg}, we additionally show the extreme event
probabilities computed using a linear parameter-to-event map, namely
$F_\gamma$ linearized around $B^\star$, the LDT-optimizer for
$\lambda=12$. When the parameter-to-event map is linear, the extreme
event set is a half-space over which we can integrate the rate
function exactly. The resulting values shown in \Cref{fig:prob reg}
underestimate the extreme event probability and results in an
incorrect asymptotic rate. This highlights the role of the
nonlinearity in the parameter-to-event map.

\begin{figure}[tbhp]\centering
	\begin{tikzpicture}[spy using outlines={rectangle,lens={scale=2}, size=8cm, connect spies}]
	\begin{semilogyaxis}[compat=newest, width=8cm, height=5.5cm, scale only axis,
	xlabel={Threshold $z$ [m]},
	ylabel=Probability $\mathbb{P}(F_\gamma\geq z)$,
	xmin=0, xmax=1.2, ymin=1e-14,ymax=1,
	legend style={font=\tiny,nodes=right}]
	\addplot [color=\cMC, line width=1]
	table[x=zz,y=mean_prob] {\data/Toho_Approx_meanstd10nos20.txt};
	\addlegendentry{MC sampling}
	\addplot [color=\cMCci, line width=0.8,dashed, forget plot]
	table[x=zz,y=mean_prob_CI] {\data/Toho_Approx_mean_CIstd10nos20.txt};
	\addplot [color=\cMCci, line width=0.8,dashed]
	table[x=zz,y=mean_prob_CI] {\data/Toho_Approx_mean_CI2std10nos20.txt};
	\addlegendentry{MC 95\% CI}
	\addplot [color=\cSORM, mark=*, mark size=1.5pt,
line width=0.5, mark options={solid}]
		table[x=F,y=prob] {\data/Toho_Approx_SORMstd10nos20.txt};
		\addlegendentry{Second-order approx.}
	\addplot [color=\cIS, mark=none, mark size=3pt, line width=1,densely dashdotted]
	table[x=zz,y=mean_prob] {\data/Toho_Approx_IS_meanstd10nos20.txt};
	\addlegendentry{Importance sampling}
	\addplot [color=\cISci, line width=0.8,densely dotted, forget plot]
	table[x=zz,y=mean_prob_CI] {\data/Toho_Approx_IS_mean_CIstd10nos20.txt};
	\addplot [color=\cISci, line width=0.8,densely dotted]
	table[x=zz,y=mean_prob_CI] {\data/Toho_Approx_IS_mean_CI2std10nos20.txt};
	\addlegendentry{IS 95\% CI}
        \coordinate (spypoint) at (0.61,3e-5);
\coordinate (magnifyglass) at (0.24,1e-10);
\end{semilogyaxis}
\spy [gray, size=2.5cm] on (spypoint)
in node[fill=white] at (magnifyglass);
\end{tikzpicture}
	\caption{Comparison of estimation using IS for regularized
          objective $F_\gamma$ with $\gamma=0.003$.  In green we show
          the mean and 95\% confidence intervals obtained with IS. The
          results obtained with standard MC sampling and second-order
          approximation of $\varOmega(z)$ are as in \cref{fig:prob
            reg} and shown for comparison.  For IS, the same LDT
          minimizers for different values of $\lambda$ as for the
          second-order approximation are used. We use 100 samples for
          each LDT-optimizer to estimate the probability following
          \eqref{eq:PE-P(z) IS}. For other values of $z$, we use the
          samples at the nearest minimizer to estimate the
          probability.  As can be seen, the IS results align well with
          the results from the second-order
          approximation.}\label{fig:IS reg}
\end{figure}
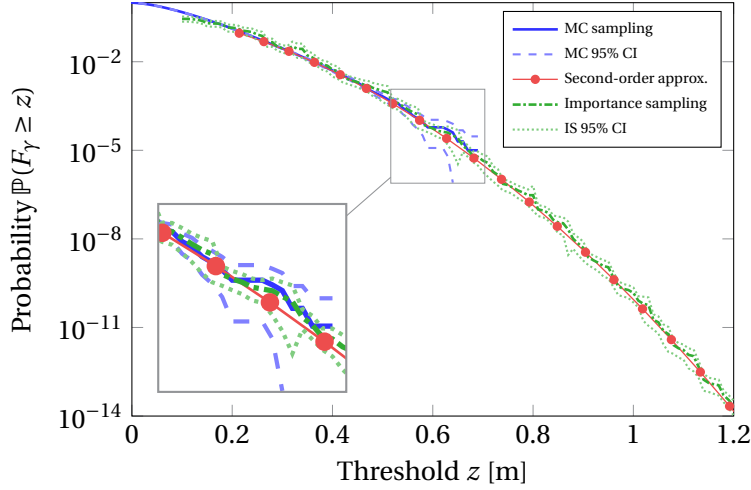

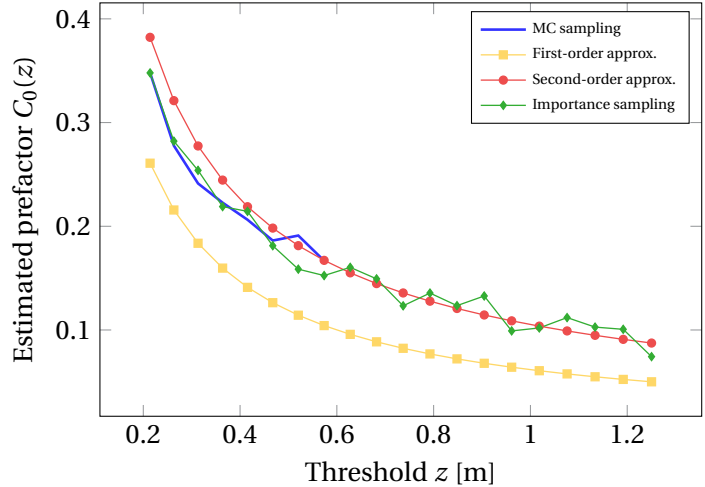
\begin{figure}[tbhp]\centering
	\begin{tikzpicture}[
	]
	\begin{axis}[compat=newest, width=8cm, height=5.5cm, scale only axis,
	xlabel={Threshold $z$ [m]},
	ylabel=Estimated prefactor $C_0(z)$,
	legend style={font=\tiny,nodes=right}]
	\addplot [color=\cMC, line width=1]
	table[x=zz,y=mean_prob] {\data/C0_Approx_meanstd10nos20.txt};
	\addlegendentry{MC sampling}
	\addplot [color=\cFORM, mark=square*, mark size=1.5pt, line
	width=0.5, mark options={solid}]
	table[x=F,y=prob] {\data/C0_Approx_FORMstd10nos20.txt};
	\addlegendentry{First-order approx.}
	\addplot [color=\cSORM, mark=*, mark size=1.5pt,
	line width=0.5, mark options={solid}]
	table[x=F,y=prob] {\data/C0_Approx_SORMstd10nos20.txt};
	\addlegendentry{Second-order approx.}
		\addplot [color=\cIS, mark=diamond*, mark size=1.5pt, line width=0.5,
		]
	table[x=zz,y=mean_prob] {\data/C0_Approx_ISstd10nos20.txt};
	\addlegendentry{Importance sampling}
	\end{axis}
	\end{tikzpicture}
	\caption{Comparison of estimated prefactor for regularized
          objective $F_\gamma$ \eqref{eq:AP-F regularized} with
          $\gamma=0.003$ as also uses in \Cref{fig:prob reg,fig:IS
            reg}. Shown is the estimated prefactor
          $C_o(z)=\PP(F_\gamma\geq z)/\exp(-I(\thetastar(z)))$ with
          \numosamples samples of vanilla MC (blue), and estimations
          using an first-order and second-order approximation of
          $\varOmega(z)$ (\cref{sec:FORM,sec:SORM}) in red and yellow,
          respectively. Each marker represents the solution of an LDT
          optimization problem with a different value of
          $\lambda$. The green line shows the IS estimation of
          $C_0$. Since only $10^3$ samples are used for each
          optimizer, sampling error is still visible.
        }\label{fig:C0}
\end{figure}

The results obtained with IS are shown in \Cref{fig:IS reg}. For each
$\lambda$ also used in \Cref{fig:prob reg}, we use 100 samples from the
shifted distribution centered at the optimizer $B^\star$ to compute
\eqref{eq:PE-IS estimate P^IS Gaussian} at $z=z(\lambda)$, and in a
neighborhood. Note that IS based on the shift of the mean is efficient even 
for large $z$, which correspond to extreme events. Despite only
using 100 samples, we obtain tight $95\%$ confidence intervals.  We only
show the results for the regularized objective $F_\gamma$, but IS
applies analogously to $F$ using the time-optimal optimizers, and
we have obtained similar results. In particular, IS with 100 samples has
comparable accuracy as SORM in
\Cref{fig:prob tim}.

To make the comparison between the different methods easier, we
compare results obtained with different methods for estimating the
prefactor $C_0(z)$ in \Cref{fig:C0}. As can be seen, the second-order
approximation of $C_0(z)$ converges to the prefactor estimated using
IS as $z$ increases, which demonstrates that the second-order
approximation \eqref{eq:PE-SORM Gauss} is an asymptotic estimation of
the original probability $P(z)$ as discussed in \cref{sec:SORM}.
In contrast, the first-order estimation of $C_0(z)$ does not converges to
the IS estimated prefactor, demonstrating that the correction factor
computed by the second-order approximation is crucial. These
observations are consistent with our discussion in
\cref{sec:SORM,sec:FORM}.

\section{Discussions and conclusions}
In this paper, we use arguments from LDT to relate probability
estimation of extreme events to optimization problems. These
optimization problems typically involve solving a PDE, and thus we
apply the adjoint method to compute derivatives efficiently in a
manner that is independent of the parameter space
dimension. Additionally, we observe numerically that the number of
iterations required to solve these LDT optimization problems is
insensitive to the extremeness of the event.  If the underlying
parameter distribution is a multivariate Gaussian distribution, the
LDT-prefactor required for the probability estimate can be computed
using either (1) a second-order approximations of the extreme event
set boundary combined with a randomized SVD or (2) importance sampling
with a proposal centered at the LDT optimizer.  We observe that the
cost of these methods is either independent or depends only
weakly on how extreme the event is. Moreover, it is independent of the
discretization dimensions.  This is a significant improvement over MC
methods whose performance typically suffers from the parameter
dimension and the level of extremeness of the event.  Since the method
based on the second-order set approximation appears to be accurate and
does not require MC sampling, it might be a good candidate for
applications where the target is the control or mitigation of extreme
events.

Our main application is a 1D tsunami problem, which is a
simplification from realistic two-dimensional tsunamis.  It is
definitely interesting to expand this application to 2D.  The main
methods including the optimization formulation from LDT and the
approximation using first/second-order information will remain as in
1D. The main challenges are technical, i.e., modeling tsunami waves
and a realistic bathymetry in 2D, and deriving and implementing the
corresponding adjoint equations.

\appendix

\section{Examples of rate functions $I(\theta)$ for different
	distributions $\mu(\theta)$} \label{sec:rate-examples}
Here, we provide examples of the derivation of rate functions for different distributions.
\begin{example}[Multivariate normal distribution in $\RR^n$]\label{ex:gaussian}
	Consider a multivariate random variable $\theta\sim\mathcal{N}(\theta_0,C)$. The cumulant generating function $S(\eta)$ is 
	\begin{equation}
	\label{eq:LDT-Gaussian S}
	\begin{aligned}
	S(\eta)=&\log\int_{\varOmega}e^{\langle\eta,\theta\rangle}(2\pi)^{-\frac{n}{2}}(\det C)^{-\frac{1}{2}}e^{-\frac{1}{2}(\theta-\theta_0)^\top C^{-1}(\theta-\theta_0)}d\theta\\
	=&\log\left[e^{\eta^\top \theta_0+\frac{1}{2}\eta^\top C\eta}\cdot \int_{\varOmega}(2\pi)^{-\frac{n}{2}}(\det C)^{-\frac{1}{2}}  e^{-\frac{1}{2}(\theta-\theta_0-C\eta)^\top C^{-1}(\theta-\theta_0-C\eta)}d\theta\right] \\
	=&\log\left[e^{\eta^\top \theta_0+\frac{1}{2}\eta^\top C\eta}\cdot 1 \right] 
	=\eta^\top \theta_0+\frac{1}{2}\eta^\top C\eta.
	\end{aligned}
	\end{equation}
	Thus, the rate function $I(\theta)$ for a multivariate Gaussian distribution is 
	\begin{equation}
	\label{eq:LDT-Gaussian I}
	\begin{aligned}
	I(\theta)=&\max_{\eta\in\RR^n}\left(\eta^\top \theta- \eta^\top \theta_0-\frac{1}{2}\eta^\top C\eta\right)\\
	=&\left[C^{-1}(\theta-\theta_0)\right]^\top (\theta-\theta_0) -\frac{1}{2}\left[C^{-1}(\theta-\theta_0)\right]^\top C[C^{-1}(\theta-\theta_0)]
	=\frac{1}{2}\|\theta-\theta_0\|_{\invC}^2,
	\end{aligned}
	\end{equation}
	since the maximum is obtained at $\eta=C^{-1}(\theta-\theta_0)$.
	Thus, $I(\theta)$ is, up to a normalization constant, the
	negative log-probability density of $\theta$. Hence, for a Gaussian
	distribution, the LDT optimization problem \eqref{eq:LDT-instanton} is
	finding the most probable point, i.e., the point maximizing the
	log-density.
\end{example}

While in this paper we focus on finite dimensional random variables,
we show that the previous example generalizes to Gaussian random
fields.
\begin{example}[Gaussian random field]\label{ex:GRF}
	Assume that the parameter is a Gaussian random field
	$\theta(x)\Gauss{\theta_0(x)}{\mathcal{C}}$. Here, $\mathcal{C}$ is
	a trace-class covariance operator defined over a Hilbert space
	$\varOmega$. For instance, $\varOmega=L^2(\mathcal{D})$ for a
	physical domain $\mathcal{D}\subset \RR^n$, $n\in \{1,2,3\}$, and
	thus each sample $\theta$ is a real-valued function over
	$\mathcal{D}$.  An example for such a covariance operator is
	$\mathcal{C}=(-\Delta+\gamma I)^{-2}, \gamma>0$, with appropriate
	boundary conditions.  The parameter $\theta(x)$ has the
        Karhunen-Lo\`eve expansion
	$\theta(x)=\theta_0(x)+\sum_{j=1}^\infty\sqrt{\lambda_j}\xi_je_j(x),
	\ x\in \mathcal{D}$,
	where $\xi_j$ are independent standard normal variables
	$\xi_j\Gauss{0}{1}$, and $\lambda_j> 0$ 
	, $e_j$ are eigenvalues
	and orthonormal eigenfunctions of $\mathcal{C}$, i.e.,
	$\mathcal{C}e_j=\lambda_je_j$  \cite{le2010spectral}.
	Let $\eta\in \varOmega$, then
	$\langle\eta,\theta\rangle=\langle\eta,\theta_0\rangle+\sum_{j=1}^\infty\sqrt{\lambda_j}\xi_j\langle\eta,e_j\rangle$.
	For the cumulant generating function $S(\eta)$, we obtain
	\begin{equation*}
	\begin{aligned}
	S(\eta)=&\log\int_\varOmega e^{
		\langle\eta,\theta_0\rangle+\sum\limits_{j=1}^\infty\sqrt{\lambda_j}\xi_j\langle\eta,e_j\rangle}d\mu(\theta)
        = \log\left(e^{\langle\eta,\theta_0\rangle}
	\prod_{j=1}^\infty\int_\RR
	e^{\sqrt{\lambda_j}\xi_j\langle\eta,e_j\rangle }
	e^{-\frac{1}{2}\xi_j^2}d\xi_j \right)
	\\
	=&\langle\eta,\theta_0\rangle+\sum_{j=1}^\infty\log\int_\RR e^{\sqrt{\lambda_j}\xi_j\langle\eta,e_j\rangle}\frac{1}{\sqrt{2\pi}}e^{-\frac{1}{2}\xi_j^2}d\xi_j\\
	=&\langle\eta,\theta_0\rangle+\sum_{j=1}^\infty\log\left( e^{\frac{1}{2}\lambda_j\langle\eta,e_i\rangle^2}\int_\RR\frac{1}{\sqrt{2\pi}}e^{-\frac{1}{2}(\xi_j-\sqrt{\lambda_j}\xi_j\langle\eta,e_j\rangle)^2}d\xi_j\right)
	= \langle\eta,\theta_0\rangle+\sum_{j=1}^\infty\frac{1}{2}\lambda_j\langle\eta,e_j\rangle^2.
	\end{aligned}
	\end{equation*}
	The corresponding rate function $I(\theta)$  is
	\begin{equation*}
	\begin{aligned}
	I(\theta)=
	\max_{\eta\in\varOmega}\left[\<\eta,\theta\>-\left(\<\eta,\theta_0\>+\sum_{j=1}^\infty\frac{1}{2}\lambda_j\<\eta,e_j\>^2
          \right)  \right]
	\end{aligned}
	\end{equation*} 
	For any given $\theta$, the optimal $\eta$ for the above maximization problem should satisfy the first-order optimality condition, i.e., $\theta-\theta_0-\sum_{j=1}^\infty \lambda_j \<\eta,e_j\>e_j=0$. Thus, the maximum is obtained for $\eta = \sum_{j=1}^\infty \lambda_j^{-1} \<\theta-\theta_0,e_j\>e_j$. Plugging in this $\eta$ and using the facts that $\theta-\theta_0=\sum_{j=1}^{\infty}\<\theta-\theta_0,e_j\>e_j$
	and $\{e_j\}$ is an eigenfunction basis of
	$\mathcal{C}$, we obtain:
	\begin{equation*}
\begin{aligned}
I(\theta)
=&\<
\sum_{i=1}^\infty\frac{1}{\lambda_i}\<\theta-\theta_0, e_i\> e_i ,\theta\>-\left(\<\sum_{i=1}^\infty \frac{1}{\lambda_i}\<\theta-\theta_0, e_i\> e_i,\theta_0\>+\sum_{j=1}^\infty\frac{1}{2}\lambda_j\< \sum_{i=1}^\infty \frac{1}{\lambda_i}\<\theta-\theta_0, e_i\> e_i,e_j\>^2
\right)  \\
=& \frac{1}{2}\sum_{j=1}^\infty \frac{1}{\lambda_j} \<\theta-\theta_0,e_j\>^2
=\frac{1}{2}\|\theta-\theta_0\|^2_{\mathcal{C}^{-1}}.
\end{aligned}
\end{equation*} 
	 The above computations only hold for $\theta$
        such that all infinite sums converge. Otherwise, we define
        $I(\theta):=\infty$.
\end{example}

\begin{example}[Exponential distribution]\label{ex:exp}
	Consider a parameter $\theta$ with $n$ independent components
	$\theta_k$'s, each of which satisfies an exponential
	distribution with $\alpha_k>0$, i.e.,
	\begin{equation}
	\label{eq:LDT-exp prob}
	d\mu(\theta) = \prod_{k=1}^n\alpha_ke^{-\alpha_k\theta_k}d\theta_k \qquad \text{for } \theta_k \geq 0.
	\end{equation}
	The corresponding cumulant generating function $S(\eta)$ is
	\begin{equation}
	\label{eq:LDT-exp S}
	\begin{aligned}
	S(\eta)=\log\prod_{k=1}^n\int_0^\infty e^{\eta_k \theta_k}\alpha_ke^{-\alpha_k\theta_k}d\theta_k
	=-\sum_{k=1}^n\log\left(1-\frac{\eta_k}{\alpha_k}\right) \qquad  \text{for }\eta_k<\alpha_k.
	\end{aligned}
	\end{equation}
	The associated rate function is 
	\begin{equation}
	\label{eq:LDT-exp I}
	I(\theta)=\max_{\eta\in\RR^n, \eta_k<\alpha_k}\left[ \langle\eta,\theta\rangle+\sum_{k=1}^n\log\left(1-\frac{\eta_k}{\alpha_k}\right)\right] =\sum_{k=1}^n\left(\alpha_k\theta_k-1-\log\theta_k\right)\qquad \text{for } \theta_k>0,
	\end{equation}
	since the maximum is reached for
	$\eta_k=\alpha_k-1/\theta_k<\alpha_k$. Note that, unlike in
	the Gaussian case, $I(\theta)$ is not a multiple of the
	negative log-density. Rather, the rate function includes the
	additional terms $-1-\log(\theta_k)$ and thus a minimizer of
	the rate function $\theta^\star(z)$ might not maximize the
	density, i.e., be the most probably point.
\end{example}

\begin{example}[Other non-Gaussian distribution]
	For other non-Gaussian distributions, it may not be possible
	to derive an explicit form for the cumulant generating
	function $S(\eta)$ nor for the rate function $I(\theta)$. As a
	remedy, one could numerically approximate the rate function
	and its derivative. Alternatively, if available, one could use a mapping
	between a Gaussian distribution and the target distribution, and, for the
	LDT arguments discussed next, absorb that mapping into the
	definition of the parameter-to-event map $F$.
\end{example}

\section{Probability estimation using first-order approximation of $\varOmega(z)$}\label{sec:FORM}
In this approach, we integrate the measure $\mu(\theta)$ on the
first-order approximation of the set $\varOmega(z)$ to approximate
$P(z)$. In the engineering literature, a similar method is known as
first-order reliability method (FORM) \cite{du2001most}. We replace
$F(\theta)$ with the first-order Taylor expansions of $F(\theta)$ at
$\theta^\star$, i.e.,
\begin{equation}
\label{eq:PE-FORM F}
\begin{aligned}
F^{FO}(\theta):=F(\theta^\star(z))+\langle\dthetaF{\theta^\star(z)}, \theta-\theta^\star(z)\rangle,
\end{aligned}
\end{equation}
where $F(\theta^\star(z))=z$. Replacing the set $\varOmega(z)=\{\theta:F(\theta)\geq z \}$ with $\mathcal{H}(z):=\{\theta:F^{FO}(\theta)\geq z\}$, results in the half-space approximation $\mathcal{H}(z)$ of $\varOmega(z)$ defined in \eqref{eq:LDT-H(z)}, where $\nstar$ is the normal direction (parallel to $\dthetaF{\theta^\star}$). The corresponding first-order approximation of $P(z)$ is
\begin{equation}
\label{eq:PE-FORM}
\begin{aligned}
P^{FO}(z):=&\mu(\mathcal{H}(z))=\mu(\{\theta:\left\langle \nstar(z),\theta-\theta^\star(z)\right\rangle \geq0 \})\\
=&e^{-I(\theta^\star(z))}\int_{-\infty}^\infty
e^{-\|\eta^\star(z)\| s} \|\eta^\star(z)\| \mu_{\eta^\star(z)}\left( \mathcal{H}(z)\backslash\mathcal{H}(z,s)\right) \,ds,
\end{aligned}
\end{equation}
where the last equality follows from \eqref{eq:LDT-P(z) wrt G(z,s)}, $\mu_{\eta^\star(z)}$ is the tilted measure \eqref{eq:LDT-tilted measure}, and $\mathcal{H}(z,s)$ is the set defined in \eqref{eq:LDT-H(z,s)}. If the tilted measure on the strip $ \mathcal{H}(z)\backslash\mathcal{H}(z,s)$ is known explicitly, this allows to compute $P^{FO}(z)$.

For a multivariate Gaussian parameter, we can compute $P^{FO}(z)$ explicitly. First, we state an auxiliary result for the standard normal distribution.

\begin{lemma}[Measure of half-space for the standard normal distribution]\label{lm:FO}
	Assume given the standard normal parameter $\xi\Gauss{0}{I_n}$
	in $\RR^n$ with measure $\musn$, 
	$\xi^\star=\xisnorm\eone$ aligned with the first basis vector
	and the half-space $\Hxi:=\left\lbrace \xi:\left\langle
	\eone,\xi-\xistar \right\rangle \geq0 \right\rbrace $. Then,
	the measure $\musn(\Hxi)$ can be computed as
	\begin{equation}
	\label{eq:PE-FORM Gauss xi}
	\begin{aligned}
	\musn(\Hxi) =(2\pi)^{-1/2}\int_{\|\xistar\|}^{\infty}e^{-\frac{1}{2}s^2}ds\lesssim(2\pi)^{-1/2}}\dfrac{1}{\|\xistar\|}e^{-\frac{1}{2}\|\xistar\|^2,
	\end{aligned}
	\end{equation}
	where the asymptotic inequality holds for $\xisnorm\to\infty$.
\end{lemma}
\begin{proof}
	For every $\xi\in\Hxi$, we can split $\xi$ into two parts:
	\begin{equation}
	\label{eq:PE-xi FORM}
	\xi=\xistar+s\eone+\eortho=(\|\xistar\|+s)\eone+\eortho, \qquad s>0,\qquad \eortho\in\eospace.
	\end{equation}
	Using the orthogonality of $\eone$ and $\eortho$, and the projection $\Proj$, we find
	\begin{equation}
	\label{eq:PE-xinorm}
	\xinorm^2=(\|\xistar\|+s)^2+\|\eortho\|^2=(\|\xistar\|+s)^2+\|\Proj(\eortho)\|^2_{\RR^{n-1}}.
	\end{equation}
	Applying Fubini's theorem, the measure of the half-space $\musn(\Hxi)$ becomes
	\begin{equation}
	\label{eq:PE-FORM Gauss compute}
	\begin{aligned}
	\musn(\Hxi)=& (2\pi)^{-n/2}\int_{\Hxi}e^{-\frac{1}{2}\|\xi\|^2}d\xi= (2\pi)^{-n/2}\int_{0}^{\infty} \int_{\Proj(\eospace)} e^{-\frac{1}{2}\left[ (\|\xistar\|+s)^2+\|\Proj(\eortho)\|^2_{\RR^{n-1}}\right] }d\Proj(\eortho) ds\\
	=&(2\pi)^{-n/2}\int_{0}^{\infty}e^{-\frac{1}{2}(\|\xistar\|+s)^2}ds\int_{\mathbb{R}^{n-1}}e^{-\frac{1}{2}\|\zeta\|^2_{\RR^{n-1}}}d\zeta
	=(2\pi)^{-1/2}\int_{0}^{\infty}e^{-\frac{1}{2}(\|\xistar\|+s)^2}ds\\
	=&(2\pi)^{-1/2}\int_{\|\xistar\|}^{\infty}e^{-\frac{1}{2}s^2}ds.
	\end{aligned}
	\end{equation}	
	This proves the equality in \eqref{eq:PE-FORM Gauss xi}. The
	asymptotic estimate follows from
	\begin{equation}
	\label{eq:PE-FORM Gauss approx}
	\begin{aligned}
	\musn(\Hxi)=&(2\pi)^{-1/2}\int_{0}^{\infty}e^{-\frac{1}{2}(\|\xistar\|+s)^2}ds=(2\pi)^{-1/2}e^{-\frac{1}{2}\|\xistar\|^2}\int_{0}^{\infty}e^{-\|\xistar\|s-\frac{1}{2}s^2}ds\\
	\lesssim&(2\pi)^{-1/2}e^{-\frac{1}{2}\|\xistar\|^2}\int_{0}^{\infty}e^{-\|\xistar\|s}ds=(2\pi)^{-1/2}}\dfrac{1}{\|\xistar\|}e^{-\frac{1}{2}\|\xistar\|^2 .
	\end{aligned}
	\end{equation}
	Here, we drop the term $-\frac{1}{2}s^2$ because it is dominated by $-\|\xistar\|s$ for large $\|\xistar\|$.
\end{proof}

For the Gaussian parameter \Gaussiancase, we apply the affine transformation \eqref{eq:PE-affine transfer} to \cref{lm:FO} to obtain the explicit form of $P^{FO}(z)$ defined in \eqref{eq:PE-FORM}.
\begin{theorem}[First-order approximation for general Gaussian distributions]\label{thm:FO}
	Assume given a Gaussian parameter \Gaussiancase~and the optimizer $\thetastar(z)$ of \eqref{eq:LDT-min}. Then, the first-order approximation $P^{FO}(z)$ defined in \eqref{eq:PE-FORM} can be computed as
	\begin{equation}
	\label{eq:PE-FORM Gauss}
	P^{FO}(z)=(2\pi)^{-1/2}\int_{\sqrt{2I(\theta^\star(z))}}^{\infty}e^{-\frac{1}{2}s^2}ds\lesssim(2\pi)^{-1/2}\dfrac{1}{\sqrt{2I(\theta^\star(z))}}e^{-I(\theta^\star(z))},
	\end{equation}
	where the asymptotic estimate $\lesssim$ is for $z\to\infty$.
\end{theorem}
\begin{proof}
	Using the affine transformation \eqref{eq:PE-affine transfer} and
	\eqref{eq:PE-dF xi}, we obtain
	\begin{equation}
	\label{eq:PE-dxiF-xi}
	\begin{aligned}
	& \left\langle \nstar, \theta-\thetastar\right\rangle 
	= \left\langle\dthetaF{\thetastar}/\|\dthetaF{\thetastar}\|, \theta-\thetastar\right\rangle \\
	= &\left\langle A^{-\top}\dxiF{\xistar}/\|\dthetaF{\thetastar}\|, A\xi-A\xistar\right\rangle 
	=\frac{\|\dxiF{\xistar}\|}{\|\dthetaF{\thetastar}\|} \left\langle \eone, \xi-\xistar \right\rangle.
	\end{aligned}
	\end{equation}
	Thus, the affine transformation of the half-space $\mathcal{H}(z)$ becomes
	\begin{equation}
	\left\lbrace \xi: \frac{\|\dxiF{\xistar}\|}{\|\dthetaF{\thetastar}\|} \left\langle \eone, \xi-\xistar \right\rangle \geq 0\right\rbrace =\Hxi(z),
	\end{equation}
	i.e., the first-order approximation $P^{FO}(z)=\mu(\mathcal{H}(z))=\musn(\Hxi)$. Applying \cref{lm:FO} and \eqref{eq:PE-F xi} with $\xisnorm=\sqrt{2I(\thetastar)}$, we obtain
	\begin{equation}
	P^{FO}(z)=\musn(\Hxi)=(2\pi)^{-1/2}\int_{\sqrt{2I(\thetastar)}}^{\infty}e^{-\frac{1}{2}s^2}ds\lesssim(2\pi)^{-1/2}\dfrac{1}{\sqrt{2I(\thetastar)}}e^{-I(\thetastar)}.
	\end{equation}

\end{proof}

Note that the integral in \eqref{eq:PE-FORM Gauss} in \cref{thm:FO} is the CDF of the standard normal, which can be computed using the error function, i.e.,
\begin{equation}
\Phi(\alpha):=(2\pi)^{-1/2}\int_{-\alpha}^{\infty}e^{-\frac{1}{2}s^2}ds=\frac{1}{2}\left[1+\text{erf}\left( \frac{\alpha}{\sqrt{2}}\right) \right]\qquad \text{~for~} \alpha<0.
\end{equation}
The right estimate in \cref{thm:FO} also provides an
asymptotic approximation of $P^{FO}(z)$, which suggests 
that 
the prefactor is
$C_0(z)=(2\pi)^{-1/2}/\sqrt{2I(\theta^\star(z))}$. However, the error of this prefactor is not controllable, the asymptotic estimation of the probability we should use is the second-order approximation \eqref{eq:PE-SORM Gauss}, as discussed in \cref{sec:SORM}.

\section*{Acknowledgments}
We appreciate helpful discussions with Randall LeVeque, Marsha Berger,
Jonathan Weare, Gregor Gassner and Stefan Ulbrich. We would like to
thank the anonymous referees for their thoughtful comments and
suggestions that helped us improve our paper. We also thank Elisabeth Ullmann and Jules Pertinand for discussions on the additional assumptions required for \cref{lm:SO,thm:SO}.

\bibliographystyle{siamplain}
\bibliography{Methodology}

\end{document}